\documentclass[12pt,a4paper]{necessary}

\usepackage{amsmath,amssymb,amsfonts,amstext,amsthm,amscd}
\usepackage{mathtools}
\usepackage{mathdots}
\numberwithin{equation}{section}

\usepackage[numbers]{natbib}
\usepackage{hyperref}

\usepackage[vcentermath]{youngtab}
\usepackage{ytableau}
\usepackage[all]{xy}

\usepackage{tikz}
\usetikzlibrary{
  calc,
  intersections,
  patterns,
  decorations.markings,
  shapes.geometric,
  arrows
}

\usepackage{enumitem}
\usepackage{diagbox}
\usepackage{multirow,multicol}
\usepackage{pdflscape}
\usepackage{indentfirst}
\usepackage{setspace}
\usepackage{color}
\usepackage{colortbl}

\setlist[description]{leftmargin=\parindent,labelindent=\parindent}

\definecolor{mygray}{gray}{0.80}
\definecolor{myblue}{rgb}{0.50,0.50,0.90}
\definecolor{myred}{rgb}{0.95,0.70,0.30}

\let\OLDthebibliography\thebibliography
\renewcommand{\thebibliography}[1]{%
  \OLDthebibliography{#1}%
  \setlength{\parskip}{0pt}%
  \setlength{\itemsep}{0.5pt}%
}

\newtheorem{theorem}{Theorem}[section]
\newtheorem{theorem*}{Theorem}
\newtheorem{corollary}[theorem]{Corollary}
\newtheorem{corollary*}[theorem*]{Corollary}
\newtheorem{lemma}[theorem]{Lemma}
\newtheorem{proposition}[theorem]{Proposition}

\theoremstyle{definition}
\newtheorem{definition}[theorem]{Definition}

\newtheorem*{question*}{Question}

\newtheorem*{problem*}{Problem}

\newtheorem*{conjecture*}{Conjecture}
\newtheorem{example}[theorem]{Example}

\newtheorem*{notation*}{Notation}

\newtheorem*{claim*}{Claim}

\newtheorem{definition-proposition}[theorem]{Definition-Proposition}

\DeclareMathOperator{\moduleCategory}{\mathsf{mod}}
\renewcommand{\mod}{\moduleCategory}

\DeclareMathOperator{\proj}{\mathsf{proj}}

\DeclareMathOperator{\Hom}{\mathrm{Hom}}

\DeclareMathOperator{\thick}{\mathsf{thick}}
\DeclareMathOperator{\add}{\mathsf{add}}
\DeclareMathOperator{\Fac}{\mathsf{Fac}}

\newcommand{\twosilt}{\mathsf{2\mbox{-}silt}\hspace{.02in}}

\newcommand{\tautilt}{\mbox{\sf $\tau$-tilt}\hspace{.02in}}
\newcommand{\stautilt}{\mbox{\sf s$\tau$-tilt}\hspace{.02in}}
\newcommand{\twosiltep}{\mathsf{2\mbox{-}silt}_{\epsilon}\hspace{.02in}}
\newcommand{\twosiltno}{\mathsf{2\mbox{-}silt}}

\newcommand{\Db}{\mathsf{D^b}}
\newcommand{\Kb}{\mathsf{K^b}}
\newcommand{\RHom}{\mathbf{R}\strut\kern-.2em\operatorname{Hom}\nolimits}

\begin{document}

\setlength{\baselineskip}{17pt}

\title{\textbf{A complete classification of $\tau$-tilting finite Schur algebras}}

\author{Toshitaka Aoki and Qi Wang}


\abstract{\indent
In this paper, we determine the $\tau$-tilting finiteness for some blocks of classical Schur algebras. Combining with the results in \cite{W-schur}, we get a complete classification of $\tau$-tilting finite Schur algebras. We also give a complete classification of $\tau$-tilting finite blocks of the Schur algebra $S(2,r)$. As an application, we obtain a classification of strictly wild Schur algebras, except for three remaining open cases.
}

\maketitle

\begin{center}
\end{center}

\section{Introduction}

Let $n,r$ be positive integers and $\mathbb{F}$ an algebraically closed field of characteristic $p>0$.
We take an $n$-dimensional vector space $V$ over $\mathbb{F}$ with a basis
$\{v_1,v_2,\ldots,v_n\}$.
We denote by $V^{\otimes r}$ the $r$-fold tensor product
$V\otimes_\mathbb{F}V\otimes_\mathbb{F}\cdots\otimes_\mathbb{F}V$.
Then, $V^{\otimes r}$ has an $\mathbb{F}$-basis given by
\[
\left\{
v_{i_1}\otimes v_{i_2}\otimes \cdots\otimes v_{i_r}
\mid
1\leqslant i_j\leqslant n \ \text{for all}\ 1\leqslant j\leqslant r
\right\}.
\]
Let $G_r$ be the symmetric group on $r$ symbols and $\mathbb{F}G_r$ its group algebra.
Then, $G_r$, and hence also $\mathbb{F}G_r$, act on the right on $V^{\otimes r}$ by placing permutations of the subscripts, that is, for any $\sigma \in G_r$,
\[
(v_{i_1}\otimes v_{i_2}\otimes \cdots\otimes v_{i_r})\cdot \sigma
=
v_{i_{\sigma(1)}}\otimes v_{i_{\sigma(2)}}\otimes \cdots\otimes v_{i_{\sigma(r)}}.
\]
We call the endomorphism ring
$\mathrm{End}_{\mathbb{F}G_r}(V^{\otimes r})$
the \emph{Schur algebra} (see \cite[Section 2]{Martin-schur alg}) and denote it by
$S_\mathbb{F}(n,r)$, or simply by $S(n,r)$.

One of the important roles of $S(n,r)$ is that the module category of $S(n,r)$ is equivalent to the category of $r$-homogeneous polynomial representations of the general linear group $\mathrm{GL}_n(\mathbb{F})$.
Based on this well-known fact, $q$-Schur algebras, infinitesimal Schur algebras, Borel-Schur algebras, etc., appear as derivatives.
In addition, $S(n,r)$ is known to be a quasi-hereditary algebra that is closely related to the highest weight category and plays a significant role in Lie theory.
Consequently, the representation theory of $S(n,r)$ has been well-studied in the past few decades.
We refer to \cite{Donkin-Schur-IV}, \cite{DEMN-tame schur}, \cite{DN-semisimple}, \cite{EH-two-blocks}, \cite{EH-method}, \cite{EL-Torsion-pair-schur}, \cite{Martin-schur alg}, etc., for more properties of Schur algebras.

In the representation theory of finite-dimensional algebras, the notion of progenerators is crucial because it characterizes the equivalence between the module categories of two algebras, in which case these two algebras are called Morita equivalent.
In the 1970s, a generalization called tilting modules, together with tilting complexes, was discovered during the development of Morita theory for derived categories and has been extensively studied by many mathematicians, such as \cite{BB-tilting-module}, \cite{HU-tilting}, \cite{R-derived}, \cite{RS-tilting-complex}, etc.
In recent years, the $\tau$-tilting theory introduced by Adachi-Iyama-Reiten \cite{AIR} has drawn more and more attention, where $\tau$ denotes the Auslander-Reiten translation.
A central notion of $\tau$-tilting theory is the class of \emph{support $\tau$-tilting modules}, which is a completion to the class of tilting modules from the viewpoint of mutation.
Moreover, support $\tau$-tilting modules are in bijection with several important objects in representation theory, such as two-term silting complexes, functorially finite torsion classes, and left finite semibricks.
One may look at \cite{AI-silting}, \cite{Asai}, \cite{AHMW-joint}, \cite{BST-max-green-seq}, \cite{DIRRT} and \cite{EJR18} for more materials.

Similar to the representation-finiteness of algebras, a modern analog called \emph{$\tau$-tilting finiteness} was introduced by Demonet-Iyama-Jasso \cite{DIJ-tau-tilting-finite}.
A finite-dimensional algebra $A$ is called \emph{$\tau$-tilting finite} if it has only finitely many isomorphism classes of basic support $\tau$-tilting modules.
In the $\tau$-tilting finite case, the set of support $\tau$-tilting modules has a nice behavior, for example, the set admits bijections to the set of all torsion classes \cite{DIJ-tau-tilting-finite} and the set of all semibricks \cite{Asai}.
So far, the $\tau$-tilting finiteness for many algebras is known by several papers, such as \cite{Ada-rad-square-0}, \cite{AAC-brauer grapha}, \cite{Mizuno}, \cite{MS-cycle-finite}, \cite{Mousavand-biserial alg}, \cite{P-gentle}, \cite{W-two-point}, \cite{W-simply}, \cite{Z-tilted}, etc.

Since the representation type of $S(n,r)$ over an algebraically closed field $\mathbb{F}$ of characteristic $p>0$ is completely determined, we may consider the $\tau$-tilting finiteness of $S(n,r)$ as a new property.
Another motivation can be found in \cite{AS-Schurian-finiteness}, where the authors are working on the $\tau$-tilting finiteness for block algebras of type A Hecke algebras.
In fact, a complete classification of $\tau$-tilting finite Schur algebras could provide useful material for such a classification of Hecke algebras.

We mention that the $\tau$-tilting finiteness of most Schur algebras has been determined in \cite{W-schur}, and the only remaining cases are displayed in $(\star)$ below.
Hence our first aim in this paper is to solve these open cases.
\[
(\star)\quad
\left\{
\begin{aligned}
p&=2,       & n&=2,          & r&=8,17,19; \\
p&=2,       & n&=3,          & r&=4;\\
p&=2,       & n&\geqslant 5, & r&=5; \\
p&\geqslant 5, & n&=2,       & p^2&\leqslant r\leqslant p^2+p-1.
\end{aligned}
\right.
\]

\begin{theorem*}[Theorem \ref{thm:classification Schur}]
\label{thm:tau-tilting finiteness}
Let $S(n,r)$ be the Schur algebra over $\mathbb{F}$.
\begin{description}[itemsep=-3pt]
  \item[(1)] If $p=2$, then $S(2,8)$, $S(2,17)$ and $S(2,19)$ are $\tau$-tilting finite.
  \item[(2)] If $p=2$, then $S(3,4)$ is $\tau$-tilting finite.
  \item[(3)] If $p=2$, then $S(n,5)$ is $\tau$-tilting infinite for any $n\geqslant 5$.
  \item[(4)] If $p\geqslant 5$, then $S(2,r)$ is $\tau$-tilting finite for any $p^2\leqslant r\leqslant p^2+p-1$.
\end{description}
\end{theorem*}

According to the work in \cite{W-schur} and the above results, we have obtained a complete classification of $\tau$-tilting finite Schur algebras.
In Appendix \ref{appendix}, we provide a complete list of $\tau$-tilting finite Schur algebras.
Also, the number of isomorphism classes of basic support $\tau$-tilting modules for $\tau$-tilting finite $S(n,r)$'s over $p=2,3$ is completely determined.

We obtain the explicit relation between representation-finiteness and $\tau$-tilting finiteness for $S(n,r)$ in Appendix \ref{appendix}.
In particular, we have
\begin{corollary*}
Suppose $p=2$, $n\geqslant 3$ or $p=3$ or $p\geqslant 5$, $n\geqslant 2$.
Then, $S(n,r)$ is $\tau$-tilting finite if and only if it is representation-finite.
\end{corollary*}

In order to understand a $\tau$-tilting finite Schur algebra $S(n,r)$, it is enough to find the basic algebra $\overline{S(n,r)}$ of $S(n,r)$.
We point out that finding the quiver and relations of $\overline{S(n,r)}$ for a $\tau$-tilting finite $S(n,r)$ has been completed in several previous works, including \cite{Erdmann-finite}, \cite{DEMN-tame schur} and \cite{W-schur}.
In the proof of Theorem \ref{thm:tau-tilting finiteness}, one finds that, for example, if $p=2$,
\[
\overline{S(2,8)}\simeq \mathcal{L}_5,\quad
\overline{S(2,17)}\simeq \mathcal{L}_5\oplus \mathcal{A}_2 \oplus \mathbb{F}\oplus \mathbb{F},\quad
\overline{S(2,19)}\simeq \mathcal{L}_5\oplus \mathcal{D}_3 \oplus \mathbb{F} \oplus \mathbb{F},
\]
where $\mathcal{A}_m$, $\mathcal{D}_m$ and $\mathcal{L}_5$ are defined in Subsection \ref{subsection-all algebras}.
We notice that the $\tau$-tilting finiteness of $S(n,r)$ is always reduced to that of blocks of $S(n,r)$, which is the main strategy to prove Theorem \ref{thm:tau-tilting finiteness}.
However, the $\tau$-tilting infiniteness of $S(n,r)$ does not imply the $\tau$-tilting infiniteness of the blocks of $S(n,r)$.
In other words, it is possible that a $\tau$-tilting infinite $S(n,r)$ has a $\tau$-tilting finite block.
This motivates us to try to give a classification of $\tau$-tilting finite blocks of Schur algebras.

\begin{problem*}
\label{problem}
Give a complete classification of $\tau$-tilting finite blocks of Schur algebras.
\end{problem*}

The second aim of this paper is to give a partial answer to the above problem.
Namely, we determine all $\tau$-tilting finite blocks of $S(2,r)$ as follows.

\begin{theorem*}[Theorem \ref{thm:S_2r} and Section \ref{subsection-all algebras}]
\label{thm:1.3}
Let $\mathcal{B}$ be a block of $S(2,r)$.
\begin{description}[itemsep=-3pt]
  \item[(1)] If $p=2$, then $\mathcal{B}$ is $\tau$-tilting finite if and only if $\mathcal{B}$ is Morita equivalent to one of $\mathbb{F}$, $\mathcal{A}_2$, $\mathcal{D}_3$, $\mathcal{K}_4$ and $\mathcal{L}_5$.
  \item[(2)] If $p\geqslant 3$, then $\mathcal{B}$ is $\tau$-tilting finite if and only if $\mathcal{B}$ is Morita equivalent to one of $\mathbb{F}$, $\mathcal{A}_m$ $(2\leqslant m \leqslant p)$ and $\mathcal{D}_{p+1}$.
\end{description}
\end{theorem*}

In order to prove Theorem \ref{thm:1.3}, a crucial statement we use, see \cite{EH-two-blocks}, is that if two blocks, $\mathcal{B}$ of $S(2,r)$ and $\mathcal{B}'$ of $S(2,r')$, have the same number of simple modules over the same field, then $\mathcal{B}$ and $\mathcal{B}'$ are Morita equivalent.
However, such a phenomenon does not appear in the case of $S(n,r)$ with $n \geqslant 3$.
One may easily find a counterexample in Appendix \ref{appendix}, for example, comparing $S(3,4)$ and $S(3,5)$ over $p=2$.

A study on the $\tau$-tilting finiteness of $q$-Schur algebras $S_q(n,r)$ is in progress and will be settled in another article.
The conclusions in this paper have laid the foundation for the study of $q$-Schur algebras, at least in the case of $S_q(2,r)$.
More precisely, it is shown in \cite[Lemma 3.2]{EN-q-Schur} that each $S_q(2,r)$ is Morita equivalent to a direct sum of some $S(2,r')$'s over the same field; once we know the precise summands of the direct sum, it is obvious to find the $\tau$-tilting finiteness of $S_q(2,r)$.

This paper is organized as follows.
In Section 2, we first review some basic concepts of $\tau$-tilting theory and silting theory.
Then, we introduce the sign decomposition, which is the main tool we use in this paper to show the $\tau$-tilting finiteness of $\mathcal{D}_m$.
Last, we recall the definitions and properties of Schur algebras.
In Section 3, we present our main results and proofs.
In Section 4, we provide a classification of strictly wild Schur algebras.

\section{Preliminaries}

We recall that any finite-dimensional basic algebra $A$ over an algebraically closed field $\mathbb{F}$ is isomorphic to a bound quiver algebra $\mathbb{F}Q/I$, where $\mathbb{F}Q$ is the path algebra of a finite quiver $Q$ and $I$ is an admissible ideal of $\mathbb{F}Q$.
We call $Q$ the quiver of $A$.
We refer to \cite{ASS} for more background on quiver representation theory and the representation theory of finite-dimensional algebras.

\subsection{$\tau$-tilting theory}
\label{subsection-tau-tilting-theory}

Let $\moduleCategory A$ be the category of finitely generated right $A$-modules and $\proj A$ the category of finitely generated projective right $A$-modules.
For any $M\in \moduleCategory A$, we denote by $\add(M)$, respectively $\Fac(M)$, the full subcategory of $\moduleCategory A$ whose objects are direct summands, respectively factor modules, of finite direct sums of copies of $M$.
We denote by $A^{\mathrm{op}}$ the opposite algebra of $A$ and by $\lvert M\rvert$ the number of isomorphism classes of indecomposable direct summands of $M$.

We recall that the Nakayama functor $\nu:=D(-)^\ast$ is induced by the dualities
\[
D=\Hom_\mathbb{F}(-,\mathbb{F})\colon \moduleCategory A \leftrightarrow \moduleCategory A^{\mathrm{op}}
\quad\text{and}\quad
(-)^\ast=\Hom_A(-,A)\colon \proj A \leftrightarrow \proj A^{\mathrm{op}}.
\]
Then, for any $M\in \moduleCategory A$ with a minimal projective presentation
\[
P_1 \overset{f_1}{\longrightarrow}
P_0 \overset{f_0}{\longrightarrow}
M \longrightarrow 0,
\]
the Auslander-Reiten translation $\tau M$ is defined by the exact sequence
\[
0 \longrightarrow \tau M
\longrightarrow \nu P_1
\overset{\nu f_1}{\longrightarrow}
\nu P_0.
\]

\begin{definition}[{\cite[Definition 0.1]{AIR}}]
\label{def-tau-tilting}
Let $M\in \moduleCategory A$.
\begin{description}[itemsep=-3pt]
  \item[(1)] $M$ is called \emph{$\tau$-rigid} if $\Hom_A(M,\tau M)=0$.
  \item[(2)] $M$ is called \emph{$\tau$-tilting} if $M$ is $\tau$-rigid and $\lvert M\rvert=\lvert A\rvert$.
  \item[(3)] $M$ is called \emph{support $\tau$-tilting} if $M$ is a $\tau$-tilting $(A/\langle e\rangle)$-module with respect to an idempotent $e$ of $A$.
\end{description}
\end{definition}

A pair $(M,P)$ with $M\in \moduleCategory A$ and $P\in \proj A$ is called a \emph{support $\tau$-tilting pair} if $M$ is $\tau$-rigid, $\Hom_A(P,M)=0$ and $\lvert M\rvert+\lvert P\rvert=\lvert A\rvert$.
Obviously, $(M,P)$ is a support $\tau$-tilting pair if and only if $M$ is a $\tau$-tilting $(A/\langle e\rangle)$-module and $\add(P)=\add(eA)$.

We denote by $\mathsf{\tau\text{-}rigid}\, A$ the set of isomorphism classes of $\tau$-rigid $A$-modules, and by $\stautilt A$, respectively $\tautilt A$, the set of isomorphism classes of basic support $\tau$-tilting, respectively $\tau$-tilting, modules.
It is known from \cite{AIR} that
\[
\tautilt A \subseteq \stautilt A \subseteq \mathsf{\tau\text{-}rigid}\, A.
\]

\begin{definition}
\label{def-tau-tilting-finite}
A finite-dimensional algebra $A$ is called \emph{$\tau$-tilting finite} if $\tautilt A$ is a finite set.
Otherwise, $A$ is called \emph{$\tau$-tilting infinite}.
\end{definition}

\begin{proposition}[{\cite[Corollary 2.9]{DIJ-tau-tilting-finite}}]
\label{tau-tilting-finite-rigid}
A finite-dimensional algebra $A$ is $\tau$-tilting finite if and only if one of, equivalently any of, the sets $\mathsf{\tau\text{-}rigid}\, A$ and $\stautilt A$ is a finite set.
\end{proposition}

It is shown in \cite[Section 2.2]{AIR} that every $\tau$-rigid $A$-module $M$ provides a torsion class $\Fac M$ in $\moduleCategory A$.
If moreover $A$ is $\tau$-tilting finite, then all torsion classes in $\moduleCategory A$ can be obtained in this way; see \cite[Theorem 3.8]{DIJ-tau-tilting-finite}.
According to the inclusions between torsion classes, we may define a partial order on $\stautilt A$, i.e.,
\[
M\leqslant N
\quad \Longleftrightarrow \quad
\Fac M \subseteq \Fac N
\]
for $M,N\in \stautilt A$.
We denote by $\mathcal{H}(\stautilt A)$ the Hasse quiver of $\stautilt A$ with respect to this partial order $\leqslant$.
From the viewpoint of mutation, it is known that arrows of $\mathcal{H}(\stautilt A)$ are explicitly described by left mutations of support $\tau$-tilting modules; see \cite[Section 2.4]{AIR} for details.
The following statement implies that $A$ is $\tau$-tilting finite if we can find a finite connected component in $\mathcal{H}(\stautilt A)$.

\begin{proposition}[{\cite[Corollary 2.38]{AIR}}]
\label{finite-connected-component}
If the Hasse quiver $\mathcal{H}(\stautilt A)$ contains a finite connected component $\Delta$, then $\mathcal{H}(\stautilt A)=\Delta$.
\end{proposition}

We give an example to illustrate the constructions above.

\begin{example}
\label{example-simplest}
Let
\[
A=
\mathbb{F}
\left(
\xymatrix{
1\ar@<0.5ex>[r]^{\alpha}
&
2 \ar@<0.5ex>[l]^{\beta}
}
\right)
/\langle\alpha\beta,\beta\alpha\rangle.
\]
We denote by $S_1$, $S_2$ the simple $A$-modules and by $P_1$, $P_2$ the indecomposable projective $A$-modules.
We observe that
\[
S_1\simeq
\vcenter{\xymatrix@C=0.05cm@R=0.1cm{
1
}},
\quad
S_2\simeq
\vcenter{\xymatrix@C=0.05cm@R=0.1cm{
2
}},
\quad
P_1\simeq
\vcenter{\xymatrix@C=0.05cm@R=0.1cm{
1\ar@{-}[d]\\
2
}},
\quad
P_2\simeq
\vcenter{\xymatrix@C=0.05cm@R=0.1cm{
2\ar@{-}[d]\\
1
}}.
\]
Here, we describe $A$-modules via their composition factors.
For instance, we denote the simple module $S_i$ by $i$ and then, $\substack{1\\2}=\substack{S_1\\S_2}$ is an indecomposable module $M$ with a unique simple submodule $S_2$ such that $M/S_2\simeq S_1$.
We deduce that
\begin{itemize}[itemsep=-3pt]
  \item $\tautilt A=\{P_1\oplus P_2, P_1\oplus S_1, S_2\oplus P_2\}$,
  \item $\stautilt A=\{P_1\oplus P_2, P_1\oplus S_1, S_2\oplus P_2, S_1,S_2,0\}$,
\end{itemize}
and therefore, $A$ is $\tau$-tilting finite.
The Hasse quiver $\mathcal{H}(\stautilt A)$ is given as follows:
\[
\vcenter{
\xymatrix@C=1.2cm@R=0.1cm{
&
P_1\oplus S_1 \ar[r]^-{>}
&
S_1 \ar[dr]^-{>}
\\
P_1\oplus P_2 \ar[dr]_-{>} \ar[ur]^-{>}
&&&
0
\\
&
S_2\oplus P_2 \ar[r]_-{>}
&
S_2 \ar[ur]_-{>}
&
}
}.
\]
\end{example}

We recall some reduction methods for the $\tau$-tilting finiteness of $A$.

\begin{proposition}[{\cite[Theorem 5.12]{DIRRT}, \cite[Theorem 4.2]{DIJ-tau-tilting-finite}}]
\label{quotient and idempotent}
If $A$ is $\tau$-tilting finite, then
\begin{description}[itemsep=-3pt]
  \item[(1)] the quotient algebra $A/I$ is $\tau$-tilting finite for any two-sided ideal $I$ of $A$,
  \item[(2)] the idempotent truncation $eAe$ is $\tau$-tilting finite for any idempotent $e$ of $A$.
\end{description}
\end{proposition}

\begin{proposition}[{e.g., \cite[Theorem 3.1]{Ada-rad-square-0}}]
\label{infinite-square}
Let $A=\mathbb{F}Q/I$ be a bound quiver algebra with an admissible ideal $I$.
If the quiver $Q$ contains
\[
\xymatrix@C=1cm@R=0.7cm{
\circ \ar@<0.5ex>[d] \ar@<0.5ex>[r]
&
\circ \ar@<0.5ex>[l] \ar@<0.5ex>[d]
\\
\circ \ar@<0.5ex>[u] \ar@<0.5ex>[r]
&
\circ \ar@<0.5ex>[u] \ar@<0.5ex>[l]
}
\]
as a subquiver, then $A$ is $\tau$-tilting infinite.
\end{proposition}

It is worth mentioning that Eisele, Janssens and Raedschelders \cite{EJR18} provided a powerful reduction theorem, as shown below.

\begin{proposition}[{\cite[Theorem 1]{EJR18}}]
\label{center}
Let $I$ be a two-sided ideal generated by central elements which are contained in the Jacobson radical of $A$.
Then, there exists a poset isomorphism between $\stautilt A$ and $\stautilt (A/I)$.
\end{proposition}

Lastly, the following statement is obvious.

\begin{proposition}
\label{prop:number-product}
Suppose that $A$ is decomposed into $A_1,A_2,\ldots,A_\ell$ as blocks.
Then, $A$ is $\tau$-tilting finite if and only if $A_i$ is $\tau$-tilting finite for any $1\leqslant i\leqslant \ell$.
If this is the case, then
\[
\#\stautilt A
=
\prod_{i=1}^{\ell}\#\stautilt A_i.
\]
\end{proposition}

\subsection{Silting theory}

We denote by $\mathsf{C^b}(\proj A)$ the category of bounded complexes of projective $A$-modules and by $\Kb(\proj A)$ the corresponding homotopy category which is triangulated.
We denote by $\Db(\mod A)$ the bounded derived category of the abelian category $\mod A$.

\begin{definition}[{\cite[Definition 2.1]{AI-silting}}]
A complex $T\in \Kb(\proj A)$ is called \emph{presilting} if
\[
\Hom_{\Kb(\proj A)}(T,T[i])=0
\quad
\text{for any } i>0.
\]
A presilting complex $T$ is called \emph{silting} if $\thick T=\Kb(\proj A)$, where $\thick T$ is the smallest full triangulated subcategory containing $T$ which is closed under direct summands.
Moreover, a silting complex $T$ is called \emph{tilting} if $\Hom_{\Kb(\proj A)}(T,T[i])=0$ for any $i<0$.
\end{definition}
We introduce the silting mutation of silting complexes; see \cite[Definition 2.30]{AI-silting}.
Let $T=X\oplus Y$ be a basic silting complex in $\Kb(\proj A)$ with a direct summand $X$.
We take a minimal left $\mathsf{add}(Y)$-approximation $\pi$ and a triangle
\[
X \overset{\pi}{\longrightarrow} Z \longrightarrow X' \longrightarrow X[1].
\]
Then $\mu_X^-(T):=X'\oplus Y$ is again a basic silting complex in $\Kb(\proj A)$; see \cite[Theorem 2.31]{AI-silting}.
We call $\mu_X^-(T)$ the left silting mutation of $T$ with respect to $X$.
Dually, we can define the right mutation $\mu_X^+(T)$ of $T$ with respect to $X$, which is also silting.

Recall that a complex in $\Kb(\proj A)$ is called \emph{two-term} if it is isomorphic to a complex $T$ concentrated in degrees $0$ and $-1$, i.e.,
\[
T
:=
\left(
T^{-1}\overset{d_T^{-1}}{\longrightarrow} T^0
\right)
=
\xymatrix@C=0.7cm{
\cdots \ar[r]
&
0 \ar[r]
&
T^{-1} \ar[r]^{d_T^{-1}}
&
T^0 \ar[r]
&
0 \ar[r]
&
\cdots
}.
\]
We denote by $\twosilt A$ the set of isomorphism classes of basic two-term silting complexes in $\Kb(\proj A)$.
There is a partial order $\leqslant$ on the set $\twosilt A$ which is introduced by \cite[Theorem 2.11]{AI-silting}.
For any $T,U\in\twosilt A$, we have $U\leqslant T$ if and only if $\Hom_{\Kb(\proj A)}(T,U[1])=0$.
We then denote by $\mathcal{H}(\twosilt A)$ the Hasse quiver of $\twosilt A$, which is compatible with the irreducible left mutation of silting complexes.

\begin{proposition}[{\cite[Lemma 2.25, Theorem 2.27]{AI-silting}}]
\label{prop:signcoh}
Let $T=(T^{-1}\to T^0)\in \twosilt A$.
Then, $\add A=\add(T^0\oplus T^{-1})$ and $\add T^0\cap \add T^{-1}=0$.
\end{proposition}

We explain the connection between $\tau$-tilting theory and silting theory.

\begin{theorem}[{\cite[Theorem 3.2]{AIR}}]
\label{tau-tiltiing-silting}
There exists a poset isomorphism between $\stautilt A$ and $\twosilt A$.
More precisely, the bijection is given by mapping a two-term silting complex $T$ to its $0$-th cohomology $H^0(T)$, and the inverse is given by
\[
\xymatrix@C=0.8cm@R=0.2cm{
M \ar@{|->}[r]
&
(P_1\oplus P \overset{[f,0]}{\longrightarrow} P_0)
},
\]
where $(M,P)$ is the corresponding support $\tau$-tilting pair and
\[
P_1 \overset{f}{\longrightarrow} P_0 \longrightarrow M \longrightarrow 0
\]
is the minimal projective presentation of $M$.
\end{theorem}

Suppose that $P_1,P_2,\ldots,P_n$ are pairwise non-isomorphic indecomposable projective $A$-modules.
We denote by $[P_1],[P_2],\ldots,[P_n]$ the isomorphism classes of indecomposable complexes concentrated in degree $0$.
Clearly, the classes $[P_1],[P_2],\ldots,[P_n]$ in the triangulated category $\Kb(\proj A)$ form a standard basis of the Grothendieck group $K_0(\Kb(\proj A))$.
If a two-term complex $T$ in $\Kb(\proj A)$ is written as
\[
\left(
\bigoplus_{i=1}^n P_i^{\oplus b_i}
\longrightarrow
\bigoplus_{i=1}^n P_i^{\oplus a_i}
\right),
\]
the class $[T]$ can be identified by an integer vector
\[
g(T)
:=
(a_1-b_1,a_2-b_2,\ldots,a_n-b_n)
\in \mathbb{Z}^n,
\]
which is called the \emph{$g$-vector} of $T$.
We have the following statement.

\begin{proposition}[{\cite[Theorem 5.5]{AIR}}]
\label{g-vector-injection}
The map $T\mapsto g(T)$ from $\twosilt A$ to $\mathbb{Z}^n$ is an injection.
\end{proposition}

We give an example here.

\begin{example}
\label{example-involution}
Let
\[
A=
\mathbb{F}
\left(
\xymatrix{
1\ar@<0.5ex>[r]^{\alpha}
&
2\ar@<0.5ex>[l]^{\beta}
}
\right)
/\langle\alpha\beta,\beta\alpha\rangle.
\]
In Example \ref{example-simplest}, the Hasse quiver $\mathcal{H}(\stautilt A)$ is described.
By applying Theorem \ref{tau-tiltiing-silting}, the Hasse quiver $\mathcal{H}(\twosilt A)$ is
\begin{center}
\begin{tikzpicture}[
  shorten >=1pt,
  auto,
  node distance=0cm,
  node_style/.style={font=},
  edge_style/.style={draw=black}
]
\node[node_style] (v) at (-1,0) {$
\left[
\begin{smallmatrix}
0\longrightarrow P_1\\
\oplus\\
0\longrightarrow P_2
\end{smallmatrix}
\right]
$};

\node[node_style] (v1) at (2,-1) {$
\left[
\begin{smallmatrix}
P_1\overset{\beta}{\longrightarrow} P_2\\
\oplus\\
0\longrightarrow P_2
\end{smallmatrix}
\right]
$};

\node[node_style] (v12) at (5,-1) {$
\left[
\begin{smallmatrix}
P_1\overset{\beta}{\longrightarrow} P_2\\
\oplus\\
P_1\longrightarrow 0
\end{smallmatrix}
\right]
$};

\node[node_style] (v2) at (2,1) {$
\left[
\begin{smallmatrix}
0\longrightarrow P_1\\
\oplus\\
P_2\overset{\alpha}{\longrightarrow} P_1
\end{smallmatrix}
\right]
$};

\node[node_style] (v21) at (5,1) {$
\left[
\begin{smallmatrix}
P_2\longrightarrow 0\\
\oplus\\
P_2\overset{\alpha}{\longrightarrow} P_1
\end{smallmatrix}
\right]
$};

\node[node_style] (v0) at (8,0) {$
\left[
\begin{smallmatrix}
P_1\longrightarrow 0\\
\oplus\\
P_2\longrightarrow 0
\end{smallmatrix}
\right]
$};

\draw[->] (v) edge (v1);
\draw[->] (v1) edge (v12);
\draw[->] (v12) edge (v0);
\draw[->] (v) edge (v2);
\draw[->] (v2) edge (v21);
\draw[->] (v21) edge (v0);
\end{tikzpicture}.
\end{center}
One can use this to verify Proposition \ref{prop:signcoh}.
We conclude that the set of $g$-vectors is
\[
\{(1,1),(2,-1),(1,-2),(-1,2),(-2,1),(-1,-1)\},
\]
where the $g$-vectors are illustrated in $\mathbb{Z}^2$ as follows:
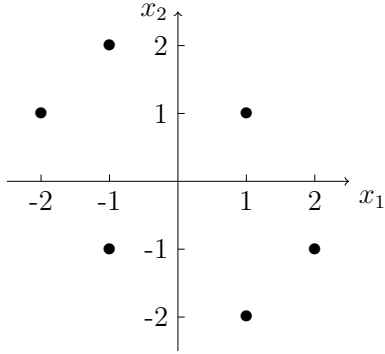
\begin{center}
\scalebox{0.9}{
\begin{tikzpicture}
  \draw[->] (-2.5,0)--(2.5,0) node[below right] {$x_1$};
  \foreach \x in {-2,-1,1,2}
    \draw (\x,0.1)--(\x,0) node[below] {$\x$};

  \draw[->] (0,-2.5)--(0,2.5) node[left] {$x_2$};
  \foreach \x in {-2,-1,1,2}
    \draw (0.1,\x)--(0,\x) node[left] {$\x$};

  \node at (1,1) {$\bullet$};
  \node at (2,-1) {$\bullet$};
  \node at (1,-2) {$\bullet$};
  \node at (-2,1) {$\bullet$};
  \node at (-1,2) {$\bullet$};
  \node at (-1,-1) {$\bullet$};
\end{tikzpicture}
}.
\end{center}
\end{example}

\subsection{Sign decomposition}
\label{sec:sign-decomp}

In this subsection, we introduce the notion of \emph{sign decomposition} of the set $\twosilt A$, which is originally considered in \cite{Aoki} as a means to classify torsion classes in radical square zero algebras, and then generalized in \cite{AHIKM} where the authors study fans and polytopes in $\tau$-tilting theory.
The main effect of sign decomposition is that the restriction of the set of $g$-vectors for the original algebra to each orthant can be described by the set of $g$-vectors for a simpler algebra.

We define $[m,n]:=\{m,m+1,\ldots,n\}$ for two integers $1\leqslant m\leqslant n$.
Recall that $A=\mathbb{F}Q/I$ is a bound quiver algebra, where the vertex set of $Q$ is given by $[1,n]=\{1,\ldots,n\}$.
We denote by $e_i$ the primitive idempotent of $A$ corresponding to the vertex $i$, and by $P_i:=e_iA$ the corresponding indecomposable projective $A$-module.

We denote by
\[
\mathsf{M}(n)
:=
\{\epsilon=(\epsilon(1),\ldots,\epsilon(n)) \colon [1,n]\to \{\pm1\}\}
\]
the set of all maps from $[1,n]$ to $\{\pm1\}$, which endows an involution $\epsilon\mapsto -\epsilon$ given by $(-\epsilon)(i)=-\epsilon(i)$ for all $i\in [1,n]$.
We sometimes simply use $\epsilon(i)=+$, respectively $\epsilon(i)=-$, to indicate $\epsilon(i)=+1$, respectively $\epsilon(i)=-1$, without causing confusion.
For each $\epsilon\in \mathsf{M}(n)$, let $\mathbb{Z}_{\epsilon}^n$ be the area determined by the linear inequality $\epsilon(i)x_i>0$ for all $i\in [1,n]$, that is,
\[
\mathbb{Z}_{\epsilon}^n
:=
\{x=(x_1,\ldots,x_n) \in \mathbb{Z}^n
\mid
\epsilon(i)x_i>0,\ i\in [1,n]\}.
\]
For example, we have $\mathsf{M}(2)=\{(+,+),(+,-),(-,+),(-,-)\}$ and $\mathbb{Z}_{\epsilon}^2$ with $\epsilon\in\mathsf{M}(2)$ is illustrated as follows:
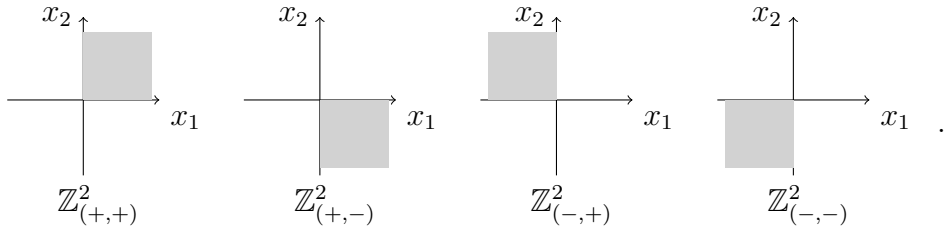
\begin{center}
\begin{tabular}{cccc}
\begin{tikzpicture}
  \draw[->] (-1,0)--(1,0) node[below right] {$x_1$};
  \draw[->] (0,-1)--(0,1.1) node[left] {$x_2$};
  \fill[gray!35] (0,0) rectangle (0.9,0.9);
  \draw (0.2,-1.4) node {$\mathbb{Z}_{(+,+)}^2$};
\end{tikzpicture}
&
\begin{tikzpicture}
  \draw[->] (-1,0)--(1,0) node[below right] {$x_1$};
  \draw[->] (0,-1)--(0,1.1) node[left] {$x_2$};
  \fill[gray!35] (0,0) rectangle (0.9,-0.9);
  \draw (0.2,-1.4) node {$\mathbb{Z}_{(+,-)}^2$};
\end{tikzpicture}
&
\begin{tikzpicture}
  \draw[->] (-1,0)--(1,0) node[below right] {$x_1$};
  \draw[->] (0,-1)--(0,1.1) node[left] {$x_2$};
  \fill[gray!35] (0,0) rectangle (-0.9,0.9);
  \draw (0.2,-1.4) node {$\mathbb{Z}_{(-,+)}^2$};
\end{tikzpicture}
&
\begin{tikzpicture}
  \draw[->] (-1,0)--(1,0) node[below right] {$x_1$};
  \draw[->] (0,-1)--(0,1.1) node[left] {$x_2$};
  \fill[gray!35] (0,0) rectangle (-0.9,-0.9);
  \draw (0.2,-1.4) node {$\mathbb{Z}_{(-,-)}^2$};
\end{tikzpicture}
\end{tabular}.
\end{center}

We consider a subset of $\twosilt A$ which plays a central role in sign decomposition.
For each $\epsilon\in \mathsf{M}(n)$, we define
\[
\twosiltep A
:=
\{T\in \twosilt A \mid g(T)\in \mathbb{Z}_{\epsilon}^n\}.
\]
More generally, for a given subset $\mathsf{M}$ of $\mathsf{M}(n)$, we define
\[
\twosiltno_{\mathsf{M}} A
:=
\bigcup_{\epsilon\in \mathsf{M}}\twosiltep A.
\]
Note that the union is disjoint by the definition of $\mathbb{Z}_{\epsilon}^n$.

\begin{proposition}
\label{prop:sgndecomp}
We have $\twosilt A=\twosiltno_{\mathsf{M}(n)} A$.
\end{proposition}

\begin{proof}
It is obvious that the set on the right is contained in the set on the left.
We assume that $T=(T^{-1}\to T^0)\in \twosilt A$.
By Proposition \ref{prop:signcoh}, each indecomposable projective $A$-module $P_i$ lies in precisely one of $\add T^0$ and $\add T^{-1}$.
In other words, the $i$-th entry of $g(T)$ is either positive or negative for all $1\leqslant i\leqslant n$.
Hence, there must exist a map $\epsilon\in \mathsf{M}(n)$ such that $T\in \twosiltep A$.
\end{proof}

\begin{proposition}[{\cite[Theorem 2.14]{AIR}}]
\label{prop:opposite}
Let $\mathsf{M}$ be a subset of $\mathsf{M}(n)$ and $-\mathsf{M}:=\{-\epsilon \mid \epsilon\in \mathsf{M}\}$.
Then the duality $(-)^\ast:=\Hom_A(-,A)$ gives a bijection
\[
\twosiltno_{\mathsf{M}} A
\overset{1-1}{\longleftrightarrow}
\twosiltno_{-\mathsf{M}} A^{\sf op}.
\]
\end{proposition}

Next, we associate an upper triangular matrix algebra $A_{\epsilon}$ to each $\epsilon\in \mathsf{M}(n)$.

\begin{definition}
\label{def:sd}
For an arbitrary map $\epsilon\in \mathsf{M}(n)$, let
\[
e_{\epsilon,+}
:=
\sum_{\epsilon(i)=+} e_i
\quad
\text{and}
\quad
e_{\epsilon,-}
:=
\sum_{\epsilon(i)=-} e_i.
\]
We define an upper triangular matrix algebra
\[
A_{\epsilon}
:=
\begin{pmatrix}
\dfrac{e_{\epsilon,+}Ae_{\epsilon,+}}{J_{\epsilon,+}}
&
e_{\epsilon,+}Ae_{\epsilon,-}
\\
0
&
\dfrac{e_{\epsilon,-}Ae_{\epsilon,-}}{J_{\epsilon,-}}
\end{pmatrix},
\]
where $J_{\epsilon,+}$, respectively $J_{\epsilon,-}$, is a two-sided ideal of $e_{\epsilon,+}Ae_{\epsilon,+}$, respectively $e_{\epsilon,-}Ae_{\epsilon,-}$, consisting of all $x$ such that $x$ is in the Jacobson radical and $xy=0$, respectively $yx=0$, for all $y\in e_{\epsilon,+}Ae_{\epsilon,-}$.
\end{definition}

We denote by $e'_1,e'_2,\ldots,e'_n$ the primitive idempotents of $A_{\epsilon}$, which are naturally corresponding to $e_1,e_2,\ldots,e_n$ of $A$, respectively.
Under this correspondence, we can regard that the quiver of $A_{\epsilon}$ has $[1,n]$ as the vertex set.
If we set $P_i':=e_i'A_{\epsilon}$ for $i\in [1,n]$, the correspondence $[P_i]\mapsto [P_i']$ induces an isomorphism
\[
K_0(\Kb(\proj A))
\longrightarrow
K_0(\Kb(\proj A_{\epsilon}))
\]
of Grothendieck groups.
Based on the construction of the algebra $A_{\epsilon}$, we have
\[
\Hom_A(e_{\epsilon,-}A,e_{\epsilon,+}A)
=
e_{\epsilon,+}Ae_{\epsilon,-}
\simeq
e'_{\epsilon,+}A_{\epsilon}e'_{\epsilon,-}
=
\Hom_{A_{\epsilon}}(e'_{\epsilon,-}A_{\epsilon},e'_{\epsilon,+}A_{\epsilon}),
\]
where
\[
e'_{\epsilon,+}:=\sum_{\epsilon(i)=+} e'_i
\quad
\text{and}
\quad
e'_{\epsilon,-}:=\sum_{\epsilon(i)=-} e'_i.
\]
This gives rise to a bijection
\begin{center}
$\{T=(T^{-1}\to T^0)\in \Kb(\proj A)\mid g(T)\in \mathbb{Z}_{\epsilon}^{n}\}$
\end{center}
\begin{center}
$\overset{1-1}{\longleftrightarrow} \{U=(U^{-1}\to U^0)\in \Kb(\proj A_{\epsilon})\mid g(U)\in \mathbb{Z}_{\epsilon}^{n}\}$
\end{center}
over two-term complexes in the homotopy categories.
Moreover, if we restrict our interest to two-term silting complexes, then all information about $\twosiltep A$ can be obtained from the algebra $A_{\epsilon}$, as shown below.

\begin{proposition}[{\cite[Theorem 4.23]{AHIKM}}]
\label{prop:Aepsilon}
For $\epsilon\in \mathsf{M}(n)$, we have an isomorphism
\[
\twosiltep A
\overset{\sim}{\longrightarrow}
\twosiltep A_{\epsilon},
\]
which preserves the $g$-vectors of two-term silting complexes.
In particular, $A$ is $\tau$-tilting finite if $A_{\epsilon}$ is $\tau$-tilting finite for all $\epsilon\in \mathsf{M}(n)$.
\end{proposition}

One may use Example \ref{example-involution} to understand the above isomorphism.

\begin{example}
Recall that
\[
A=
\mathbb{F}
\left(
\xymatrix{
1\ar@<0.5ex>[r]^{\alpha}
&
2\ar@<0.5ex>[l]^{\beta}
}
\right)
/\langle\alpha\beta,\beta\alpha\rangle
\]
and the set of $g$-vectors of two-term silting complexes is given in Example \ref{example-involution}.
We observe that the algebras $A_{\epsilon}$ for $\epsilon\in \mathsf{M}(2)=\{(+,+),(+,-),(-,+),(-,-)\}$ are given by
\[
A_{(+,+)}
=
\mathbb{F}(1\quad 2),
\quad
A_{(+,-)}
=
\mathbb{F}(1\rightarrow 2),
\quad
A_{(-,+)}
=
\mathbb{F}(1\leftarrow 2),
\quad
A_{(-,-)}
=
\mathbb{F}(1\quad 2).
\]
The corresponding subsets $\{g(T)\mid T\in \twosiltep A_{\epsilon}\}$ are given by
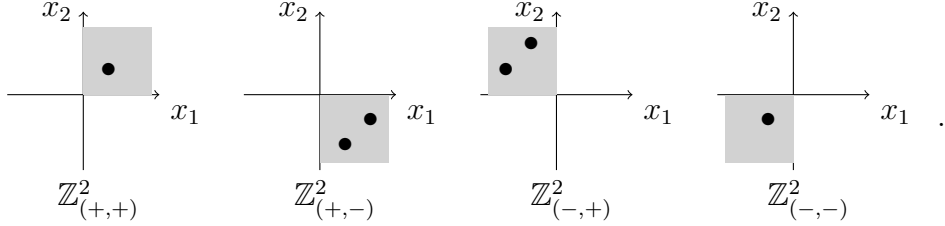
\begin{center}
\begin{tabular}{cccc}
\begin{tikzpicture}
  \draw[->] (-1,0)--(1,0) node[below right] {$x_1$};
  \draw[->] (0,-1)--(0,1.1) node[left] {$x_2$};
  \fill[gray!35] (0,0) rectangle (0.9,0.9);
  \draw (0.2,-1.4) node {$\mathbb{Z}_{(+,+)}^2$};
  \node at (1/3,1/3) {$\bullet$};
\end{tikzpicture}
&
\begin{tikzpicture}
  \draw[->] (-1,0)--(1,0) node[below right] {$x_1$};
  \draw[->] (0,-1)--(0,1.1) node[left] {$x_2$};
  \fill[gray!35] (0,0) rectangle (0.9,-0.9);
  \draw (0.2,-1.4) node {$\mathbb{Z}_{(+,-)}^2$};
  \node at (2/3,-1/3) {$\bullet$};
  \node at (1/3,-2/3) {$\bullet$};
\end{tikzpicture}
&
\begin{tikzpicture}
  \draw[->] (-1,0)--(1,0) node[below right] {$x_1$};
  \draw[->] (0,-1)--(0,1.1) node[left] {$x_2$};
  \fill[gray!35] (0,0) rectangle (-0.9,0.9);
  \draw (0.2,-1.4) node {$\mathbb{Z}_{(-,+)}^2$};
  \node at (-2/3,1/3) {$\bullet$};
  \node at (-1/3,2/3) {$\bullet$};
\end{tikzpicture}
&
\begin{tikzpicture}
  \draw[->] (-1,0)--(1,0) node[below right] {$x_1$};
  \draw[->] (0,-1)--(0,1.1) node[left] {$x_2$};
  \fill[gray!35] (0,0) rectangle (-0.9,-0.9);
  \draw (0.2,-1.4) node {$\mathbb{Z}_{(-,-)}^2$};
  \node at (-1/3,-1/3) {$\bullet$};
\end{tikzpicture}
\end{tabular}.
\end{center}
\end{example}

We explain that sign decomposition is compatible with the tilting mutation of $A$ regarded as a tilting complex concentrated in degree $0$.
We fix an integer $j\in [1,n]$ and denote by $T:=\mu^-_{P_j}(A)$ the left silting mutation of $A$ with respect to $(0\rightarrow P_j)$ for the indecomposable projective $A$-module $P_j$.
We observe that $T$ is of the form $\bigoplus_{i=1}^n T_i$ with $T_i:=(0\rightarrow P_i)$ for $i\neq j$ and $T_j=(T_j^{-1}\to T_j^0)$ is an indecomposable two-term complex in $\Kb(\proj A)$ such that $\add(P_j)=\add(T_j^{-1})$.

Suppose that $T$ is a tilting complex.
By Rickard's theorem \cite{R-derived}, the endomorphism algebra $\mathrm{End}_{\Kb(\proj A)}(T)$, written as $\mu_j^-(A)$, is derived equivalent to $A$.
More precisely, we have a triangle equivalence
\begin{equation}
\label{eq:trieq}
F\colon
\Db(\mod A)
\overset{\sim}{\longrightarrow}
\Db(\mod \mu_j^-(A))
\end{equation}
mapping $T\mapsto \mu_j^-(A)$ so that each indecomposable direct summand $T_i$ of $T$ is sent to the corresponding indecomposable projective $\mu_j^-(A)$-module $P'_i$.
Besides, it naturally induces an isomorphism
\[
K_0(\Kb(\proj A))
\overset{\sim}{\longrightarrow}
K_0(\Kb(\proj \mu_j^-(A)))
\]
of Grothendieck groups by $[T_i]\mapsto [P_i']$ for all $1\leqslant i\leqslant n$.
In this way, we can identify the vertex set of the quiver of $\mu_j^-(A)$ with $[1,n]$, so that $\mathsf{M}(n)$ is compatible between $A$ and $\mu_j^-(A)$.

Let $\mathsf{M}(n)_{j,-}$, respectively $\mathsf{M}(n)_{j,+}$, be the subset of $\mathsf{M}(n)$ consisting of all maps $\epsilon$ satisfying $\epsilon(j)=-$, respectively $\epsilon(j)=+$.
We have the following statement.

\begin{proposition}
\label{prop:sd-irrtilting}
The triangle equivalence in \eqref{eq:trieq} induces a bijection
\begin{equation}
\label{eq:posetmutation}
\twosiltno_{\mathsf{M}(n)_{j,-}} A
\overset{1-1}{\longleftrightarrow}
\twosiltno_{\mathsf{M}(n)_{j,+}} \mu_j^-(A).
\end{equation}
\end{proposition}

We point out that this bijection has been shown in \cite[Lemma 4.6]{AMN} for a class of algebras named \emph{Brauer tree algebras}, and the proof can be directly generalized to any finite-dimensional algebra $A$.
For the convenience of readers, we give a proof here.

\begin{proof}[Proof of Proposition \ref{prop:sd-irrtilting}]
Recall that $\mu_{P_j}^-(A)$, respectively $\mu_{P_j}^+(A)$, denotes the left, respectively right, silting mutation of $A$ with respect to $(0\rightarrow P_j)$.
Similar to \cite[Lemma 4.5]{AMN}, one can show that
\begin{equation}
\label{eq:mu1}
\twosiltno_{\mathsf{M}(n)_{j,-}} A
=
\{U\in \twosilt A
\mid
A[1]\leqslant U\leqslant \mu_{P_j}^-(A)\},
\end{equation}
and
\begin{equation}
\label{eq:mu2}
\twosiltno_{\mathsf{M}(n)_{j,+}} A
=
\{V\in \twosilt A
\mid
\mu_{P_j[1]}^+(A[1])\leqslant V\leqslant A\},
\end{equation}
with respect to the partial order $\leqslant$ on $\twosilt A$.

Suppose that $\mu_{P_j}^-(A)$ is tilting and $\mu_j^-(A)=\mathrm{End}_{\Kb(\proj A)}(\mu_{P_j}^-(A))$ as defined before.
We denote by $P'_i$ $(1\leqslant i\leqslant n)$ the indecomposable projective modules of $\mu_j^-(A)$.
Then, the triangle equivalence $F$ in \eqref{eq:trieq} satisfies
\[
F(\mu_{P_j}^-(A)) \simeq \mu_j^-(A)
\quad
\text{and}
\quad
F(A[1]) \simeq \mu_{P'_j[1]}^+(\mu_j^-(A)[1]).
\]
Since the triangle equivalence $F$ preserves $\leqslant$ and mutations, it gives an isomorphism
\begin{center}
$\{U\in \twosilt A \mid A[1] \leqslant U \leqslant \mu_{P_j}^- (A) \}$
\end{center}
\begin{center}
$\overset{\sim}{\longrightarrow}
\{V\in \twosilt \mu_j^-(A) \mid \mu_{P'_j[1]}^+ (\mu_{j}^-(A)[1]) \leqslant V \leqslant \mu_j^- (A)\}$.
\end{center}
We then get the desired bijection by \eqref{eq:mu1} and \eqref{eq:mu2}.
\end{proof}

The above result can be generalized as follows.
Let $J$ be a subset of $[1,n]$ and $e_J:=\sum_{j\in J} e_j$.
We denote by $T:=\mu_{P_J}^-(A)$ the left silting mutation of $A$ with respect to $(0\rightarrow P_J)$ for the projective $A$-module $P_J:=e_JA$.
By the definition of silting mutation, $T$ is of the form $\bigoplus_{i=1}^n T_i$, where $T_i:=(0\rightarrow P_i)$ for all $i\notin J$ and $T_i=(T_i^{-1}\to T_i^0)$ is an indecomposable two-term complex such that $\add(P_i)=\add(T_i^{-1})$ for every $i\in J$.

We assume that $T$ is tilting.
We denote by $\mu_J^-(A):=\mathrm{End}_{\Kb(\proj A)}(T)$ the endomorphism algebra of $T$, and by $P'_i$ the indecomposable projective $\mu_J^-(A)$-module corresponding to $T_i$ for all $1\leqslant i\leqslant n$.
Then, we can naturally identify the vertex set of the quiver of $\mu_J^-(A)$ with $[1,n]$, so that $\mathsf{M}(n)$ is compatible between $A$ and $\mu_J^-(A)$.

Let $\mathsf{M}(n)_{J,-}$, respectively $\mathsf{M}(n)_{J,+}$, be the subset of $\mathsf{M}(n)$ consisting of all maps $\epsilon$ satisfying $\epsilon(J)=\{-\}$, respectively $\epsilon(J)=\{+\}$.

\begin{corollary}
\label{prop:sd-tilting}
Under the above setting, we have a bijection
\begin{equation}
\label{eq:posetmutationJ}
\twosiltno_{\mathsf{M}(n)_{J,-}} A
\overset{1-1}{\longleftrightarrow}
\twosiltno_{\mathsf{M}(n)_{J,+}} \mu_J^-(A).
\end{equation}
In particular, the above sets are finite if one of $A$ and $\mu_J^-(A)$ is $\tau$-tilting finite.
\end{corollary}

\begin{proof}
Note that Proposition \ref{prop:sd-irrtilting} is precisely the case when $J=\{j\}$, and a proof of this statement can be done in a similar way.
We omit the details.
\end{proof}

At the end of this subsection, it is better to give an example.

\begin{example}
Let
\[
A=
\mathbb{F}
\left(
\xymatrix@C=0.7cm{
1
&
\ar[l]_{\alpha} 2 \ar[r]^{\beta}
&
3
}
\right)
\]
be a path algebra.
We denote by $P_i$ the indecomposable projective $A$-modules corresponding to the vertices $i=1,2,3$.
If we choose $J=\{1,3\}$, then $P_J=P_1\oplus P_3$.
The left silting mutation $T:=\mu_{P_J}^-(A)$ of $A$ with respect to $(0\rightarrow P_J)$ is given by $T=T_1\oplus T_2\oplus T_3$, where
\[
T_1:=(P_1\overset{\alpha\cdot}{\longrightarrow} P_2),
\quad
T_2:=(0\longrightarrow P_2),
\quad
T_3:=(P_3\overset{\beta\cdot}{\longrightarrow} P_2).
\]
It is not difficult to check that $T$ is a tilting complex.
In addition, the endomorphism algebra of $T$, written as $\mu_J^-(A)$, is
\vspace{-0.4cm}
\[
\begin{pmatrix}
\mathbb{F} & \mathbb{F} & 0 \\
0 & \mathbb{F} & 0 \\
0 & \mathbb{F} & \mathbb{F}
\end{pmatrix},
\]
and it is isomorphic to the path algebra
\[
\mathbb{F}
\left(
\xymatrix@C=0.7cm{
1 \ar[r]
&
2
&
\ar[l] 3
}
\right).
\]
We describe the bijection in Corollary \ref{prop:sd-tilting} between the Hasse quivers as follows:
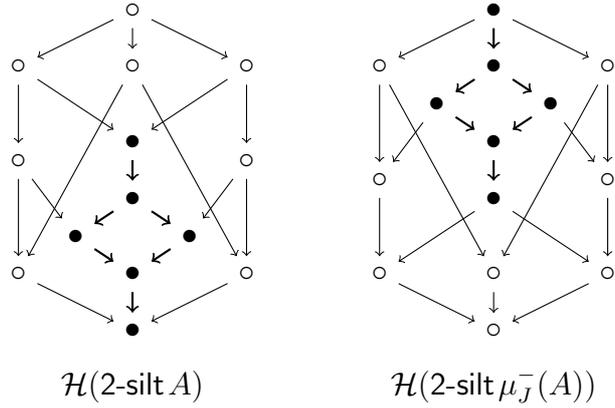
\begin{center}
\begin{tabular}{cccc}
\begin{tikzpicture}
  \node (1) at (0,7.25) {$\circ$};
  \node (2) at (-1.5,6.5) {$\circ$};
  \node (3) at (0,6.5) {$\circ$};
  \node (4) at (1.5,6.5) {$\circ$};
  \node (5) at (-1.5,5.25) {$\circ$};
  \node (6) at (1.5,5.25) {$\circ$};
  \node (7) at (0,5.5) {$\bullet$};
  \node (8) at (0,4.75) {$\bullet$};
  \node (9) at (-0.75,4.25) {$\bullet$};
  \node (10) at (0.75,4.25) {$\bullet$};
  \node (11) at (-1.5,3.75) {$\circ$};
  \node (12) at (0,3.75) {$\bullet$};
  \node (13) at (1.5,3.75) {$\circ$};
  \node (14) at (0,3) {$\bullet$};

  \draw[->] (1)--(2);
  \draw[->] (1)--(3);
  \draw[->] (1)--(4);
  \draw[->] (2)--(5);
  \draw[->] (2)--(7);
  \draw[->] (3)--(11);
  \draw[->] (3)--(13);
  \draw[->] (4)--(7);
  \draw[->] (4)--(6);
  \draw[->] (5)--(11);
  \draw[->] (5)--(9);
  \draw[->] (6)--(10);
  \draw[->] (6)--(13);
  \draw[->,thick] (7)--(8);
  \draw[->,thick] (8)--(9);
  \draw[->,thick] (8)--(10);
  \draw[->,thick] (9)--(12);
  \draw[->,thick] (10)--(12);
  \draw[->] (11)--(14);
  \draw[->,thick] (12)--(14);
  \draw[->] (13)--(14);
\end{tikzpicture}
& & &
\begin{tikzpicture}
  \node (1) at (0,-7.25) {$\circ$};
  \node (2) at (-1.5,-6.5) {$\circ$};
  \node (3) at (0,-6.5) {$\circ$};
  \node (4) at (1.5,-6.5) {$\circ$};
  \node (5) at (-1.5,-5.25) {$\circ$};
  \node (6) at (1.5,-5.25) {$\circ$};
  \node (7) at (0,-5.5) {$\bullet$};
  \node (8) at (0,-4.75) {$\bullet$};
  \node (9) at (-0.75,-4.25) {$\bullet$};
  \node (10) at (0.75,-4.25) {$\bullet$};
  \node (11) at (-1.5,-3.75) {$\circ$};
  \node (12) at (0,-3.75) {$\bullet$};
  \node (13) at (1.5,-3.75) {$\circ$};
  \node (14) at (0,-3) {$\bullet$};

  \draw[<-] (1)--(2);
  \draw[<-] (1)--(3);
  \draw[<-] (1)--(4);
  \draw[<-] (2)--(5);
  \draw[<-] (2)--(7);
  \draw[<-] (3)--(11);
  \draw[<-] (3)--(13);
  \draw[<-] (4)--(7);
  \draw[<-] (4)--(6);
  \draw[<-] (5)--(11);
  \draw[<-] (5)--(9);
  \draw[<-] (6)--(10);
  \draw[<-] (6)--(13);
  \draw[<-,thick] (7)--(8);
  \draw[<-,thick] (8)--(9);
  \draw[<-,thick] (8)--(10);
  \draw[<-,thick] (9)--(12);
  \draw[<-,thick] (10)--(12);
  \draw[<-] (11)--(14);
  \draw[<-,thick] (12)--(14);
  \draw[<-] (13)--(14);
\end{tikzpicture}
\\
$\mathcal{H}(\twosilt A)$
& & &
$\mathcal{H}(\twosilt \mu_J^-(A))$
\end{tabular}.
\end{center}
where the bullets indicate elements in $\twosiltno_{\mathsf{M}(n)_{J,-}} A$ and $\twosiltno_{\mathsf{M}(n)_{J,+}} \mu_J^-(A)$, respectively.
\end{example}

\subsection{Schur algebras}

In this subsection, we review some basics related to the representation theory of symmetric groups and Schur algebras.
One may refer to some textbooks, such as \cite{James-symmetric}, \cite{Martin-schur alg} and \cite{Sagan-symmetric}, for more details.

Let $r$ be a natural number and $\lambda=(\lambda_1,\lambda_2,\ldots)$ a sequence of non-negative integers.
We call $\lambda$ a \emph{partition} of $r$ if $\lambda_1+\lambda_2+\cdots=r$ with $\lambda_1\geqslant \lambda_2\geqslant \cdots \geqslant 0$, and the elements $\lambda_i$ are called \emph{parts} of $\lambda$.
If there exists an $n\in \mathbb{N}$ such that $\lambda_i=0$ for all $i>n$, we denote $\lambda$ by $(\lambda_1,\lambda_2,\ldots,\lambda_n)$ and call it a partition of $r$ with at most $n$ parts.
We denote by $\Omega(n,r)$ the set of all partitions of $r$ with at most $n$ parts.
It is well-known that $\Omega(n,r)$ admits the dominance order $\triangleright$ and the lexicographic order $>$; we omit the definitions.

We denote by $G_r$ the symmetric group on $r$ symbols and by $\mathbb{F}G_r$ the group algebra of $G_r$.
Each partition $\lambda=(\lambda_1,\lambda_2,\ldots,\lambda_n)$ of $r$ gives a Young subgroup $G_\lambda$ of $G_r$ defined as
\[
G_\lambda
:=
G_{\lambda_1}\times G_{\lambda_2}\times \cdots \times G_{\lambda_n}.
\]
Then the permutation $\mathbb{F}G_r$-module $M^\lambda$ is $1_{G_\lambda}\uparrow^{G_r}$, where $1_{G_\lambda}$ denotes the trivial module for $G_\lambda$ and $\uparrow$ denotes induction.
It is known that the permutation module $M^\lambda$ has a unique submodule which is isomorphic to the so-called \emph{Specht module} $S^\lambda$.
Furthermore, there is a unique indecomposable direct summand of $M^\lambda$ containing $S^\lambda$, which is called \emph{Young module} and is denoted by $Y^\lambda$.
In this way, each $M^\lambda$ with $\lambda\in \Omega(n,r)$ is a direct sum of Young modules $Y^\lambda$ with $\lambda\in \Omega(n,r)$.

Let $S(n,r)$ be the Schur algebra over an algebraically closed field $\mathbb{F}$ of characteristic $p>0$.
Since each $M^\lambda$ with $\lambda\in \Omega(n,r)$ can be regarded as a direct summand of $V^{\otimes r}$, for example, see \cite[Section 1.6]{Ve-doc-thesis}, we may construct the basic algebra $\overline{S(n,r)}$ of $S(n,r)$ as follows.
Let $B$ be a block of the group algebra $\mathbb{F}G_r$ labeled by a $p$-core $\omega$.
It is well-known that a partition $\lambda$ belongs to $B$ if and only if $\lambda$ has the same $p$-core $\omega$.
We define
\[
S_B
:=
\mathrm{End}_{\mathbb{F}G_r}
\left(
\bigoplus_{\lambda\in B\cap\Omega(n,r)} Y^\lambda
\right).
\]
Then, the basic algebra $\overline{S(n,r)}$ is given by $\bigoplus S_B$ taken over all blocks of $\mathbb{F}G_r$.
Moreover, $S_B$ is a direct sum of blocks of $\overline{S(n,r)}$.

We recall some constructions on the blocks of $S(2,r)$.
In order to avoid confusion of symbols, we use $\mathcal{B}$ to identify a block of $S(2,r)$ and we denote by $\lvert\mathcal{B}\rvert$ the number of simple $\mathcal{B}$-modules.

\begin{lemma}[{\cite[Theorem 13]{EH-two-blocks}}]
\label{EH-blocks-s(2,r)}
Let $\mathcal{B}$ and $\mathcal{B}'$ be two blocks of $S(2,r)$ and $S(2,r')$ over the same field.
If $\lvert\mathcal{B}\rvert=\lvert\mathcal{B}'\rvert$, then $\mathcal{B}$ and $\mathcal{B}'$ are Morita equivalent.
\end{lemma}

Based on the above lemma, it is useful to find the quiver of $\overline{S(2,r)}$.
Let $\lambda=(\lambda_1,\lambda_2)$ and $\mu=(\mu_1,\mu_2)$ be two partitions of $r$.
We define two non-negative integers $s:=\lambda_1-\lambda_2$ and $t:=\mu_1-\mu_2$.
We denote by $v^s$ the vertex in the quiver of $\overline{S(2,r)}$ corresponding to the Young module $Y^{(\lambda_1,\lambda_2)}$ with $s=\lambda_1-\lambda_2$.
Let $n(v^s,v^t)$ be the number of arrows from $v^s$ to $v^t$.
It is shown in \cite{EH-method} that $n(v^s,v^t)=n(v^t,v^s)$ and $n(v^s,v^t)$ is either $0$ or $1$.
We have the following recursive algorithm for computing $n(v^s,v^t)$.

\begin{lemma}[{\cite[Proposition 3.1]{EH-method}}]
\label{EH-quiver-s(2,r)}
Suppose that $p$ is a prime number and $s>t$.
Let $s=s_0+ps'$ and $t=t_0+pt'$ with $0\leqslant s_0,t_0\leqslant p-1$ and $s',t'\geqslant 0$.
\begin{description}[itemsep=-3pt]
  \item[(1)] If $p=2$, then
  \[
  n(v^s,v^t)
  =
  \begin{cases}
  n(v^{s'},v^{t'})
  &
  \text{if } s_0=t_0=1 \text{ or } s_0=t_0=0 \text{ and } s'\equiv t' \pmod{2},
  \\
  1
  &
  \text{if } s_0=t_0=0,\ t'+1=s'\not\equiv 0 \pmod{2},
  \\
  0
  &
  \text{otherwise.}
  \end{cases}
  \]

  \item[(2)] If $p>2$, then
  \[
  n(v^s,v^t)
  =
  \begin{cases}
  n(v^{s'},v^{t'})
  &
  \text{if } s_0=t_0,
  \\
  1
  &
  \text{if } s_0+t_0=p-2,\ t'+1=s'\not\equiv 0 \pmod{p},
  \\
  0
  &
  \text{otherwise.}
  \end{cases}
  \]
\end{description}
\end{lemma}

We end these preliminaries by recalling the useful classification for the representation type of Schur algebras.
In the following, some semi-simple cases are contained in the representation-finite cases.
We may distinguish the semi-simple cases following \cite{DN-semisimple}.
Namely, the Schur algebra $S(n,r)$ over a field $\mathbb{F}$ of characteristic $p>0$ is semi-simple if and only if $p>r$ or $p=2$, $n=2$, $r=3$.

\begin{proposition}[{\cite{Erdmann-finite, DEMN-tame schur}}]
\label{summary}
Let $p>0$ be the characteristic of $\mathbb{F}$.
Then, $S(n,r)$ is representation-finite if and only if $p=2$, $n=2$, $r=5,7$ or $p\geqslant 2$, $n=2$, $r<p^2$ or $p\geqslant 2$, $n\geqslant 3$, $r<2p$; 
$(\text{infinite-})$tame if and only if $p=2$, $n=2$, $r=4,9,11$ or $p=3$, $n=2$, $r=9,10,11$ or $p=3$, $n=3$, $r=7,8$.
Otherwise, $S(n,r)$ is wild.
\end{proposition}

\section{Main results}
\subsection{Some blocks of Schur algebras}\label{subsection-all algebras}
In order to give a complete classification of $\tau$-tilting finite Schur algebras, we need to recall some blocks of Schur algebras which are constructed in \cite{Erdmann-finite}, \cite{DEMN-tame schur}, \cite{Xi-schur} and \cite{W-schur}. We present these blocks here as bound quiver algebras and we will give specific references when we use them. We shall determine their $\tau$-tilting finiteness and also find the number of support $\tau$-tilting modules for $\tau$-tilting finite cases, see Proposition \ref{prop:Am}, Lemma \ref{prop:KLMN} and Proposition \ref{prop:Dm}.

We first focus on the representation-finite blocks. Let $\mathcal{A}_m:=\mathbb{F}Q/I$ ($2\leqslant m\in \mathbb{N}$) be the bound quiver algebra presented by the following quiver with relations,
\begin{equation}
\begin{aligned}
\ &\ \ \ Q: \xymatrix@C=1cm@R=0.3cm{1\ar@<0.5ex>[r]^{\alpha_1}&2\ar@<0.5ex>[l]^{\beta_1}\ar@<0.5ex>[r]^{\alpha_2}&\cdots \ar@<0.5ex>[l]^{\beta_2}\ar@<0.5ex>[r]^{\alpha_{m-2}}&m-1\ar@<0.5ex>[l]^{\beta_{m-2}}\ar@<0.5ex>[r]^{\ \ \alpha_{m-1}}&m\ar@<0.5ex>[l]^{\ \ \beta_{m-1}}}, \\ I:& \left \langle \alpha_1\beta_1,\alpha_i\alpha_{i+1},\beta_{i+1}\beta_i, \beta_i\alpha_i-\alpha_{i+1}\beta_{i+1} \mid 1\leqslant i\leqslant m-2 \right \rangle.
\end{aligned}
\end{equation}
Some related results are given as follows.
\begin{proposition}[{\cite[Theorem 2.1]{DR-finite blocks}}]\label{rep-finite-block}
Each representation-finite block of $S(n,r)$ is Morita equivalent to $\mathcal{A}_m$ for some $2\leqslant m\in \mathbb{N}$.
\end{proposition}

\begin{proposition}[{\cite[Theorem 3.2]{W-schur}, see also \cite[Theorem 5.6]{Aoki}}] \label{prop:Am}
For any positive integer $m\geqslant2$, we have $\#\stautilt\ \mathcal{A}_m = \binom{2m}{m}$.
\end{proposition}

We have known from \cite{W-schur} that all tame blocks of tame Schur algebras are $\tau$-tilting finite. Thus, we consider some wild blocks whose wildness is given in \cite{DEMN-tame schur}.
\begin{definition}\label{def:KLMN}
We define bound quiver algebras $\mathcal{K}_4$, $\mathcal{M}_4$, $\mathcal{L}_5$ and $\mathcal{N}_5$ as follows.
\begin{description}\setlength{\itemsep}{-3pt}
\item[(1)] $\mathcal{K}_4:=\mathbb{F}Q/I$ is presented by
\begin{equation}
Q:\xymatrix@C=1cm@R=0.3cm{1\ar@<0.5ex>[r]^{\alpha_1}&2\ar@<0.5ex>[l]^{\beta_1}\ar@<0.5ex>[r]^{\alpha_2}&3\ar@<0.5ex>[l]^{\beta_2}\ar@<0.5ex>[r]^{\alpha_3}&4\ar@<0.5ex>[l]^{\beta_3}}\ \text{and}\ I: \left \langle \begin{matrix}
\alpha_1\beta_1,\alpha_2\beta_2,\beta_3\alpha_3, \alpha_1\alpha_2\alpha_3, \beta_3\beta_2\beta_1,\\
\beta_1\alpha_1\alpha_2-\alpha_2\alpha_3\beta_3, \beta_2\beta_1\alpha_1-\alpha_3\beta_3\beta_2
\end{matrix}\right \rangle.
\end{equation}

\item[(2)] $\mathcal{M}_4:=\mathbb{F}Q/I$ is presented by
\begin{equation}
Q: \xymatrix@C=1.2cm@R=0.8cm{1\ar@<0.5ex>[r]^{\alpha_1 } &2\ar@<0.5ex>[r]^{\alpha_2 } \ar@<0.5ex>[l]^{\beta_1 } \ar@<0.5ex>[d]^{\alpha_3 } &3\ar@<0.5ex>[l]^{\beta_2 } \\
&4\ar@<0.5ex>[u]^{\beta_3 }   &} \ \text{and}\ I: \left \langle \begin{matrix}
\alpha_1\beta_1, \beta_3\alpha_3, \alpha_1\alpha_2,\beta_2\beta_1,\\
\alpha_1\alpha_3\beta_3, \alpha_3\beta_3\beta_1, \beta_1\alpha_1-\alpha_2\beta_2
\end{matrix}\right \rangle.
\end{equation}

\item[(3)] $\mathcal{L}_5:=\mathbb{F}Q/I$ is presented by
\begin{equation}
\begin{aligned}
\ & \ \ \ \ \ \ \ \ \ \ \ \ \ Q: \xymatrix@C=1cm@R=0.8cm{1\ar@<0.5ex>[r]^{\alpha_1} &2\ar@<0.5ex>[r]^{\alpha_2} \ar@<0.5ex>[l]^{\beta_1} \ar@<0.5ex>[d]^{\alpha_4} &4\ar@<0.5ex>[l]^{\beta_2}\ar@<0.5ex>[r]^{\alpha_3}&5\ar@<0.5ex>[l]^{\beta_3} \\
&3\ar@<0.5ex>[u]^{\beta_4}&} \ \text{and}\\
I:& \left \langle \begin{matrix}
\alpha_1\beta_1,\alpha_1\alpha_4,\beta_3\alpha_3,\beta_2\alpha_2, \beta_4\alpha_4,\beta_4\beta_1, \beta_4\alpha_2\beta_2, \alpha_1\alpha_2\alpha_3,\alpha_2\beta_2\alpha_4,\beta_3\beta_2\beta_1,\\
\beta_1\alpha_1\alpha_2-\alpha_2\alpha_3\beta_3,\beta_2\beta_1\alpha_1-\alpha_3\beta_3\beta_2,\alpha_2\beta_2\beta_1\alpha_1-\beta_1\alpha_1\alpha_2\beta_2
\end{matrix}\right \rangle.
\end{aligned}
\end{equation}

\item[(4)] $\mathcal{N}_5:=\mathbb{F}Q/I$ is presented by
\begin{equation}
\begin{aligned}
\ & \ \ \ \ \ \ \ \  \ \ \ Q: \xymatrix@C=1cm@R=0.3cm{1\ar@<0.5ex>[r]^{\alpha_1}&2\ar@<0.5ex>[l]^{\beta_1}\ar@<0.5ex>[r]^{\alpha_2}&3\ar@<0.5ex>[l]^{\beta_2}\ar@<0.5ex>[r]^{\alpha_3}&
4\ar@<0.5ex>[l]^{\beta_3}\ar@<0.5ex>[r]^{\alpha_4}&5\ar@<0.5ex>[l]^{\beta_4}} \ \text{with}\\ I:& \left \langle \begin{matrix}
\alpha_1\beta_1, \alpha_2\beta_2, \alpha_3\beta_3, \beta_4\alpha_4,\alpha_1\alpha_2\alpha_3\alpha_4, \beta_4\beta_3\beta_2\beta_1, \beta_2\alpha_2-\alpha_3\alpha_4\beta_4\beta_3, \\
\alpha_2\alpha_3\alpha_4\beta_4-\beta_1\alpha_1\alpha_2\alpha_3, \beta_3\beta_2\beta_1\alpha_1-\alpha_4\beta_4\beta_3\beta_2
\end{matrix}\right \rangle.
\end{aligned}
\end{equation}
\end{description}
\end{definition}

The $\tau$-tilting finiteness of $\mathcal{K}_4$, $\mathcal{M}_4$, $\mathcal{L}_5$ and $\mathcal{N}_5$ is declared as follows.
\begin{lemma}\label{prop:KLMN}
The algebras $\mathcal{K}_4$, $\mathcal{M}_4$, $\mathcal{L}_5$ are $\tau$-tilting finite, and $\mathcal{N}_5$ is $\tau$-tilting infinite. Furthermore, we have the following table.
\begin{center}
\begin{tabular}{|c|c|c|c|c|c|c|c|}
\hline
$A$ & $\mathcal{K}_4$ & $\mathcal{M}_4$& $\mathcal{L}_5$   \\
\hline
$\#\stautilt A$ & $136$ & $152$ & $1656$  \\ \hline
\end{tabular}
\end{center}
\end{lemma}
\begin{proof}
In order to show the numbers for $\mathcal{K}_4$, $\mathcal{M}_4$ and $\mathcal{L}_5$, it is enough to find all two-term silting complexes of them according to Theorem \ref{tau-tiltiing-silting}. This is equivalent to finding all $g$-vectors for them, since a two-term silting complex $T$ is uniquely determined by its $g$-vector $g(T)$ as explained in Proposition \ref{g-vector-injection}. To do this, we use a computer to calculate the $g$-vectors directly. We refer to (\href{https://infinite-wang.github.io/Notes/}{https://infinite-wang.github.io/Notes/}) for a complete list of $g$-vectors for $\mathcal{K}_4$, $\mathcal{M}_4$ and $\mathcal{L}_5$. Alternatively, one may verify the number for $\mathcal{K}_4$ by \cite[Proposition 4.1]{W-schur} in which the number is determined by a different way.

Let $e:=e_2+e_3+e_4$ be an idempotent of $\mathcal{N}_5$. We look at the idempotent truncation $e\mathcal{N}_5e$ and define $B$ to be the quotient algebra of $e\mathcal{N}_5e$ modulo the two-sided ideal generated by $\alpha_3$ and $\beta_2$. Then, $B$ is presented by
\begin{center}
$\mathbb{F}\left ( \xymatrix@C=1cm{\circ \ar[r]^{ }\ar@(dl,ul)^{\alpha}&\circ & \circ\ar[l]^{ } \ar@(ur,dr)^{\beta}} \right )/\langle\alpha^2, \beta^2\rangle$.
\end{center}
This is a gentle algebra and it is $\tau$-tilting infinite following \cite[Proposition 3.3]{P-gentle}. Hence, $\mathcal{N}_5$ is $\tau$-tilting infinite by Proposition \ref{quotient and idempotent} (2).
\end{proof}

Lastly, we define $\mathcal{D}_m:=\mathbb{F}Q/I$ ($m\geqslant 3$) by the following quiver and relations,
\begin{equation} \label{eq:quiverQ}
\begin{aligned}
\ & \ \  Q:\vcenter{\xymatrix@C=1.2cm@R=0.3cm{1\ar@<0.5ex>[dr]^{\alpha_1}&&&&&\\ &3\ar@<0.5ex>[ul]^{\beta_1}\ar@<0.5ex>[dl]^{\beta_2}\ar@<0.5ex>[r]^{\mu_3}&4\ar@<0.5ex>[l]^{\nu_3}\ar@<0.5ex>[r]^{\mu_4}&\cdots \ar@<0.5ex>[l]^{\nu_4}\ar@<0.5ex>[r]^{\mu_{m-2}}&m-1\ar@<0.5ex>[l]^{\nu_{m-2}}\ar@<0.5ex>[r]^{\ \ \mu_{m-1}}&m\ar@<0.5ex>[l]^{\ \ \nu_{m-1}}\\ 2\ar@<0.5ex>[ur]^{\alpha_2}&&&&&}}\\ 
I&: \left \langle \begin{matrix}
\alpha_2\beta_2,\alpha_1\beta_1,\alpha_2\beta_1\alpha_1,\beta_1\alpha_1\beta_2,\alpha_2\mu_3,\alpha_1\mu_3,\nu_3\beta_2,\nu_3\beta_1,
\mu_3\nu_3-\beta_1\alpha_1, \\ \mu_i\mu_{i+1},\nu_{i+1}\nu_i,\nu_i\mu_i-\mu_{i+1}\nu_{i+1}, 3\leqslant i\leqslant m-2
\end{matrix}\right \rangle.
\end{aligned}
\end{equation}

It is shown in \cite[3.4]{DEMN-tame schur} that $\mathcal{D}_3, \mathcal{D}_4$ are tame and $\mathcal{D}_m$ ($m\geqslant 5$) is wild. We have already known from \cite[Lemma 3.3]{W-schur} that $\mathcal{D}_3, \mathcal{D}_4$ are $\tau$-tilting finite and the numbers of support $\tau$-tilting modules are $28, 114$, respectively. More generally, we show that $\mathcal{D}_m$ is $\tau$-tilting finite for any $m\geqslant 3$.
\begin{proposition}\label{prop:Dm}
For an arbitrary integer $m\geqslant3$, the algebra $\mathcal{D}_m$ is $\tau$-tilting finite.
\end{proposition}
\begin{proof}
Let $\mathcal{D}'_m$ be the quotient algebra of $\mathcal{D}_m$ modulo the two-sided ideal generated by central elements $\beta_1\alpha_1$, $\alpha_1\beta_2\alpha_2\beta_1$ and $\nu_i\mu_i$ $(3\leqslant i\leqslant m-1)$. Then, $\mathcal{D}'_m$ is presented by the same quiver $Q$ in (\ref{eq:quiverQ}) with the following relations,
\begin{center}
$\begin{aligned}
\left \langle
\begin{matrix}
\alpha_2\beta_2,\alpha_1\beta_1,\beta_1\alpha_1,\alpha_2\mu_3,\alpha_1\mu_3,\nu_3\beta_2,\nu_3\beta_1,\alpha_1\beta_2\alpha_2\beta_1,\mu_3\nu_3-\beta_1\alpha_1, \\ \nu_{m-1}\mu_{m-1}, \mu_i\mu_{i+1},\nu_{i+1}\nu_i,\nu_i\mu_i,\mu_{i+1}\nu_{i+1}, 3\leqslant i\leqslant m-2
\end{matrix}\right \rangle.
\end{aligned}$
\end{center}
We observe that the indecomposable projective $\mathcal{D}_m'$-modules are
\begin{center}
$P'_1\simeq\vcenter{\xymatrix@C=0.05cm@R=0.1cm{
1\ar@{-}[d]\\
3\ar@{-}[d]\\
2\ar@{-}[d]\\
3\\
}}$, \quad $P'_2\simeq\vcenter{\xymatrix@C=0.05cm@R=0.1cm{
2\ar@{-}[d]\\
3\ar@{-}[d]\\
1\\
}}$, \quad $P'_3\simeq\vcenter{\xymatrix@C=0.05cm@R=0.1cm{
  &&3\ar@{-}[drr]\ar@{-}[dll]\ar@{-}[d]&\\
1 &&2\ar@{-}[d]&&4\\
  &&3\ar@{-}[d]& \\
  &&1&
}}$, \quad $P'_i\simeq\vcenter{\xymatrix@C=0.05cm@R=0.1cm{
&i\ar@{-}[ddr]\ar@{-}[ddl]&\\\\
i-1& &i+1
}}$, \quad $P'_m\simeq\vcenter{\xymatrix@C=0.05cm@R=0.1cm{
m\ar@{-}[dd]\\\\
m-1
}}$,
\end{center}
where $4\leqslant i\leqslant m-1$. Using Proposition \ref{center} with Theorem \ref{tau-tiltiing-silting}, we have a bijection
\begin{center}$
\twosilt \mathcal{D}_m \overset{1-1}{\longleftrightarrow} \twosilt \mathcal{D}'_m.
$\end{center}
In particular, $\mathcal{D}_m$ is $\tau$-tilting finite if and only if $\mathcal{D}_m'$ is $\tau$-tilting finite.

On the other hand, let $\mathcal{B}_{m}$ be an algebra defined by the same quiver $Q$ in (\ref{eq:quiverQ}) with the following relations,
\begin{center}
$\begin{aligned}
\left \langle \begin{matrix}
\alpha_2\beta_2,\alpha_1\beta_1,\alpha_2\mu_3,\alpha_1\mu_3,\nu_3\beta_2,\nu_3\beta_1, \alpha_2\beta_1\alpha_1\beta_2\alpha_2,\alpha_1\beta_2\alpha_2\beta_1\alpha_1,\\
\mu_3\nu_3-\beta_1\alpha_1\beta_2\alpha_2, \mu_3\nu_3-\beta_2\alpha_2\beta_1\alpha_1,\\
\mu_i\mu_{i+1},\nu_{i+1}\nu_i,\nu_i\mu_i-\mu_{i+1}\nu_{i+1}, 3\leqslant i\leqslant m-2
\end{matrix}\right \rangle.
\end{aligned}$
\end{center}
Then, $\mathcal{B}_{m}$ is a finite-dimensional symmetric algebra. Let $e_i$ be the primitive idempotent of $\mathcal{B}_m$ corresponding to the vertex $i\in [1,n]$ and $P_i:=e_i\mathcal{B}_m$ the indecomposable projective $\mathcal{B}_m$-modules. We observe that
\begin{center}
$P_1\simeq\vcenter{\xymatrix@C=0.05cm@R=0.05cm{
  1\ar@{-}[d]\\
  3\ar@{-}[d]\\
  2\ar@{-}[d]\\
  3\ar@{-}[d]\\
  1
}}$, \quad $P_2\simeq\vcenter{\xymatrix@C=0.05cm@R=0.05cm{
  2\ar@{-}[d]\\
  3\ar@{-}[d]\\
  1\ar@{-}[d]\\
  3\ar@{-}[d]\\
  2
}}$, \quad $P_3\simeq\vcenter{\xymatrix@C=0.05cm@R=0.05cm{
               &&3\ar@{-}[dll]\ar@{-}[d] \ar@{-}[ddrr]&\\
  1 \ar@{-}[d] &&2\ar@{-}[d]&&\\
  3 \ar@{-}[d] &&3\ar@{-}[d]&&4 \\
  2            &&1          &&  \\
               &&3 \ar@{-}[ull]\ar@{-}[u] \ar@{-}[uurr]&&
  }}$, \quad $P_i\simeq\vcenter{\xymatrix@C=0.05cm@R=0.05cm{
  &i\ar@{-}[ddr]\ar@{-}[ddl]&\\\\
  i-1& &i+1 \\ \\
  &i \ar@{-}[uur]\ar@{-}[uul]
  }}$, \quad
  $P_m\simeq\vcenter{\xymatrix@C=0.05cm@R=0.05cm{
  m\ar@{-}[dd]\\\\
  m-1 \ar@{-}[dd]\\\\
  m
  }}$,
\end{center}
where $4\leqslant i \leqslant m-1$. Since the algebra $\mathcal{D}'_m$ is precisely a quotient algebra of $\mathcal{B}_m$ modulo the two-sided ideal generated by $\beta_1\alpha_1$, $\alpha_1\beta_2\alpha_2\beta_1$ and $\nu_i\mu_i$ ($3\leqslant i \leqslant m-1$), it is enough to show that $\mathcal{B}_m$ is $\tau$-tilting finite by Proposition \ref{quotient and idempotent} (1).

For $m=3$, it is easy to check that $\mathcal{B}_3$ is $\tau$-tilting finite and $\#\stautilt \mathcal{B}_3=32$. Next, we show that $\mathcal{B}_m$ is $\tau$-tilting finite by induction on $m$. We assume that $m\geqslant 4$ and $\mathcal{B}_k$ is $\tau$-tilting finite for every $k\in [3,m-1]$. We feel free to use the sign decomposition introduced in Subsection \ref{sec:sign-decomp}. Let $\mathsf{M}:=\mathsf{M}(m)$ be the set of all maps from $[1,m]$ to $\{\pm1\}$. Suppose that $\mathsf{M}^-$ is a union of $\mathsf{M}_0^-$ and $\mathsf{M}_k^-$ with $k \in [3,m-1]$, where
\begin{itemize}\setlength{\itemsep}{-3pt}
\item $\mathsf{M}^-_0 := \{\epsilon \in \mathsf{M}\mid \epsilon(3)=-1 \text{\ and } \epsilon(j)\neq \epsilon(j+1) \text{\ for all\ } j\in [3,m-1]\}$;
\item $\mathsf{M}^-_k := \{\epsilon \in \mathsf{M} \mid \epsilon(k)=\epsilon(k+1)=-1 \text{\ and\ } \epsilon(j)\neq \epsilon(j+1) \text{\ for all\ } j\in [k+1,m-1] \}$.
\end{itemize}
We define $\mathsf{M}^+:= \{-\epsilon \mid \epsilon \in \mathsf{M}^-\}$. It is obvious that $\mathsf{M}=\mathsf{M}^- \sqcup \mathsf{M}^+$. (For example, if $m=5$, then $\mathsf{M}^-= \{(\epsilon(1),\epsilon(2), -1, +1,-1),(\epsilon(1),\epsilon(2), -1, -1,+1),(\epsilon(1),\epsilon(2), \epsilon(3), -1,-1)\}$, and the number of maps in $\mathsf{M}^-$ is 16.) Since $\mathcal{B}_m\simeq (\mathcal{B}_m)^{\sf op}$, we have a bijection
\begin{center}$
\twosiltno_{\mathsf{M}^-} \mathcal{B}_m \overset{1-1}{\longleftrightarrow} \twosiltno_{\mathsf{M}^+} \mathcal{B}_m
$\end{center}
by Proposition \ref{prop:opposite}. Then, $\mathcal{B}_m$ is $\tau$-tilting finite if and only if $\twosiltno_{\mathsf{M}^-} \mathcal{B}_m$ is finite, if and only if $\twosiltno_{\mathsf{M}_k^-} \mathcal{B}_m$ are finite for all $k \in \{0\}\cup[3,m-1]$.

(1) We fix $k\in [3,m-1]$. Let $\epsilon$ be a map in $\mathsf{M}_{k}^-$ and
\begin{center}
$\mathcal{B}:= (\mathcal{B}_m)_{\epsilon} = \begin{pmatrix}
  \frac{e_{\epsilon,+} \mathcal{B}_m e_{\epsilon,+}}{J_{\epsilon,+}} & e_{\epsilon,+} \mathcal{B}_m e_{\epsilon,-} \\
        0 & \frac{e_{\epsilon,-} \mathcal{B}_m e_{\epsilon,-}}{J_{\epsilon,-}} \end{pmatrix}$
\end{center}
as defined in Definition \ref{def:sd}. We denote by $e'_1,\ldots, e'_n$ the idempotents of $\mathcal{B}$ corresponding to $e_1,\ldots, e_n$ of $\mathcal{B}_m$, respectively. Then, we have a bijection between $\twosiltep \mathcal{B}_m$ and $\twosiltep \mathcal{B}$ by Proposition \ref{prop:Aepsilon}. Thus, we mainly observe the algebra $\mathcal{B}$ for our purpose.

Since $\epsilon(k)=\epsilon(k+1)=-1$, $e_k\mathcal{B}_me_{k+1}$ and $e_{k+1}\mathcal{B}_me_k$ are included in $e_{\epsilon,-}\mathcal{B}_me_{\epsilon,-}$, while they are also contained in $J_{\epsilon,-}$. By the definition of $\mathcal{B}$, we have $e'_k\mathcal{B}e'_{k+1} = 0 = e'_{k+1}\mathcal{B}e'_{k}$. In particular, there are no arrows between $k$ and $k+1$ in the quiver of $\mathcal{B}$. We observe that both $e':=\sum_{i=1}^ke_i$ and $f:=1-e'$ are central idempotents of $\mathcal{B}$, and hence $\mathcal{B}$ is decomposed into two blocks $\mathcal{B}':=e'\mathcal{B}e'$ and $\mathcal{B}'':=f\mathcal{B}f$. Thus, we have
\begin{equation} \label{eq:B-decomposed}
\twosiltep \mathcal{B} \overset{1-1}{\longleftrightarrow} \twosiltno_{\epsilon|_{[1,k]}} \mathcal{B}'\times \twosiltno_{\epsilon|_{[k+1,m]}}\mathcal{B}''
\end{equation}

On the one hand, we find that the block $\mathcal{B}'$ is isomorphic to $(\mathcal{B}_k)_{\epsilon|_{[1,k]}}$ for the algebra $\mathcal{B}_k$ with respect to the restriction $\epsilon|_{[1,k]}$. Then, the set $\twosiltno_{\epsilon|_{[1,k]}} \mathcal{B}'$ is bijection to $\twosiltno_{\epsilon|_{[1,k]}} \mathcal{B}_k$ by Proposition \ref{prop:Aepsilon}. Thus, $\twosiltno_{\epsilon|_{[1,k]}} \mathcal{B}'$ is finite by our induction hypothesis that $\mathcal{B}_k$ is $\tau$-tilting finite. On the other hand, since $\epsilon$ satisfies $\epsilon(k+1)=-1$ and $\epsilon(j)\neq \epsilon(j+1)$ for all $j\in [k+1,m-1]$, we have $e'_{k+\ell}\mathcal{B}e'_{k+r}\subseteq e'_{\epsilon,-}\mathcal{B}e'_{\epsilon,+} =0$ for any odd $\ell$ and even $r$ in $[1,m-k]$. Based on this fact, we deduce that the block $\mathcal{B}''$ is isomorphic to the path algebra $\mathbb{F}Q_{m-k}$ of $Q_{m-k}$, where $Q_{m-k}$ is the following quiver of Dynkin type $\mathbb{A}$ with alternating orientation:
\begin{center}
$\begin{xy}
   (0,0)*+{k+1}="1",  (20,0)*+{k+2}="2", (40,0)*+{k+3}="3",
    (52.5,0)*+{}="5", (65,0)*+{}="-2", (80,0)*+{m-1}="m-1", (95,0)*+{m}="m",
    {"2" \ar@{->} "1"}, {"2" \ar@{->} "3"}, {"3" \ar@{-}_{} "5"},
    {"5" \ar@{.}_{} "-2"},
    {"-2" \ar@{-} "m-1"},  {"m-1" \ar@{-}_{} "m"}
\end{xy}$,
\end{center}
which is $\tau$-tilting finite. (In fact, it is representation-finite.) Thus, we deduce that $\twosiltep \mathcal{B}$ is a finite set by (\ref{eq:B-decomposed}). Also, the set $\twosiltep \mathcal{B}_m$ is finite following Proposition \ref{prop:Aepsilon}. Since $\epsilon \in \mathsf{M}_{k}^-$ is arbitrary, the set $\twosiltno_{\mathsf{M}_k^-} \mathcal{B}_m$ is finite.

(2) We show that the set $\twosiltno_{\mathsf{M}_0^-} \mathcal{B}_m$ is finite. Let $J$ be the set of odd numbers in $[3,m]$. We define $\mathsf{M}_{J,-}:=\{\epsilon \in \mathsf{M}\mid \epsilon(J)=\{-1\}\}$ and $\mathsf{M}_{J,+}:=\{-\epsilon \mid \epsilon\in \mathsf{M}_{J,-}\}$. Since $\mathsf{M}_0^- \subseteq \mathsf{M}_{J,-}$ holds, it is enough to see the finiteness of the set $\twosiltno_{\mathsf{M}_{J,-}}\mathcal{B}_{m}$.

Let $\mu_{P_J}^-(\mathcal{B}_m)$ be the left silting mutation of $\mathcal{B}_m$ with respect to $(0\rightarrow P_J)$ for the projective module $P_J=\bigoplus_{j\in J}P_j$. Then, it is automatically tilting since silting complexes coincide with tilting complexes over symmetric algebras \cite[Example 2.8]{AI-silting}. Following Corollary \ref{prop:sd-tilting}, we obtain a bijection
\vspace{-0.2cm}
\begin{equation}\label{eq:tiltingB}
\twosiltno_{\mathsf{M}_{J,-}} \mathcal{B}_m  \overset{1-1}{\longleftrightarrow} \twosiltno_{\mathsf{M}_{J,+}} \mu_{J}^- (\mathcal{B}_m)
\vspace{-0.2cm}
\end{equation}
where $\mu_J^-(\mathcal{B}_m):= \mathrm{End}_{\Kb(\proj A)}(\mu_{P_J}^-(\mathcal{B}_m))$. We may look at the right hand side in (\ref{eq:tiltingB}). By direct calculation, $\mu_{P_J}^-(\mathcal{B}_m)$ is of the form $\bigoplus_{i=1}^n T_i$, where $T_i = (0\rightarrow P_i)$ for $i\notin J$ and
\begin{center}$
T_i :=\begin{cases}
    (P_3 \to P_1\oplus P_2\oplus P_4) & \text{for $i=3$}, \\
    (P_{i} \to P_{i-1}\oplus P_{i+1}) & \text{for $i= 5,7,\ldots, 2\lfloor\frac{m+1}{2}\rfloor-1$}, \\
    (P_m \to P_{m-1}) & \text{for odd $m$},
\end{cases}$
\end{center}
where $\lfloor x\rfloor$ is the largest integer less than or equal to $x$. Then, $\mu_J^-(\mathcal{B}_m)$ is presented by
\begin{center}
$\vcenter{\xymatrix@C=1.5cm@R=0.45cm{1\ar@<0.5ex>[dr]^{\alpha_1}\ar@<0.5ex>[dd]^{\alpha_3}&&&&&\\ &3\ar@<0.5ex>[ul]^{\beta_1}\ar@<0.5ex>[dl]^{\beta_2}\ar@<0.5ex>[r]^{\mu_3}&4\ar@<0.5ex>[l]^{\nu_3}\ar@<0.5ex>[r]^{\mu_4}&\cdots \ar@<0.5ex>[l]^{\nu_4}\ar@<0.5ex>[r]^{\mu_{m-2}}&m-1\ar@<0.5ex>[l]^{\nu_{m-2}}\ar@<0.5ex>[r]^{\ \ \mu_{m-1}}&m\ar@<0.5ex>[l]^{\ \ \nu_{m-1}}\\ 2\ar@<0.5ex>[ur]^{\alpha_2}\ar@<0.5ex>[uu]^{\beta_3}&&&&&}} \ \text{with}$

$\left \langle \begin{matrix}
\beta_1\alpha_3, \alpha_3\alpha_2, \alpha_2\beta_1, \alpha_1\beta_2,\beta_2\beta_3, \beta_3\alpha_1, \alpha_2\mu_3,  \alpha_1\mu_3, \nu_3\beta_1, \nu_3\beta_2\\
\alpha_1\beta_1-\alpha_3\beta_3, \alpha_2\beta_2-\beta_3\alpha_3, \beta_1\alpha_1- \mu_3\nu_3, \beta_2\alpha_2-\beta_1\alpha_1, \\
\mu_i\mu_{i+1},\nu_{i+1}\nu_i,\nu_i\mu_i-\mu_{i+1}\nu_{i+1}, 3\leqslant i\leqslant m-2
\end{matrix}\right \rangle.$
\end{center}

We describe the indecomposable projective $\mu_J^-(\mathcal{B}_m)$-modules $P''_i$ as follows,
\begin{center}
$P''_1\simeq\vcenter{\xymatrix@C=0.05cm@R=0.01cm{
  &1\ar@{-}[ddl]\ar@{-}[ddr]& \\\\
  2\ar@{-}[ddr] && 3 \ar@{-}[ddl]  \\\\
  &1 &
  }}$, \, $P''_2\simeq\vcenter{\xymatrix@C=0.05cm@R=0.01cm{
  &2\ar@{-}[ddl]\ar@{-}[ddr]& \\\\
  1\ar@{-}[ddr] && 3 \ar@{-}[ddl]  \\\\
  &2 &
  }}$, \, $P''_3\simeq\vcenter{\xymatrix@C=0.05cm@R=0.01cm{
  &3\ar@{-}[ddl]\ar@{-}[dd]\ar@{-}[ddr]& \\\\
  1\ar@{-}[ddr] & 2 \ar@{-}[dd] & 4 \ar@{-}[ddl]  \\\\
  &3 &
  }}$, \, $P''_i\simeq\vcenter{\xymatrix@C=0.05cm@R=0.01cm{
  &i\ar@{-}[ddl]\ar@{-}[ddr]& \\\\
  i-1\ar@{-}[ddr] && i+1 \ar@{-}[ddl]  \\\\
  &i &
  }}$, \, $P''_m\simeq\vcenter{\xymatrix@C=0.05cm@R=0.01cm{
  m\ar@{-}[dd]\\\\
  m-1 \ar@{-}[dd] \\\\
  m
  }}$,
\end{center}
where $4\leqslant i \leqslant m-1$. We find that $\mu_J^-(\mathcal{B}_m)$ is a symmetric algebra with radical cube zero and it is $\tau$-tilting finite following from \cite[Theorem 1.1]{AA}. In particular, both sets in (\ref{eq:tiltingB}) are finite as desired. Thus, we have shown that $\twosiltno_{\mathsf{M}_k^-} \mathcal{B}_m$ are finite for all $k \in \{0\}\cup[3,m-1]$. Therefore, $\mathcal{B}_m$ is $\tau$-tilting finite, and so are $\mathcal{D}'_m$ and $\mathcal{D}_m$.
\end{proof}

As we mentioned in Lemma \ref{prop:KLMN}, we can use a computer to find all $g$-vectors for $\mathcal{D}_m$. For the convenience of readers, we give some information as follows,
\begin{center}
\begin{tabular}{|c|c|c|c|c|c|c|c|c|}
\hline
$m $ & 3& 4 & 5& 6 &7 &8&9&10\\
\hline
$\#\twosilt \mathcal{D}_m$ & 28 & 114 &  456 & 1816&4012&13238 &45238&151568 \\ \hline
\end{tabular}.
\end{center}

\subsection{A complete classification of $\tau$-tilting finite Schur algebras} \label{sec:complete classification}
We recall that most cases are determined in \cite{W-schur} and the remaining cases are
\begin{center}
$(\star)\ \left\{\begin{aligned}
p&=2,n=2, r=8, 17, 19; \\
p&=2,n=3, r=4;\\
p&=2,n\geqslant 5,r=5; \\
p&\geqslant 5, n=2,p^2\leqslant r\leqslant p^2+p-1.
\end{aligned}\right.$
\end{center}
We have the following result. Recall that $\overline{S(n,r)}$ is the basic algebra of $S(n,r)$.
\begin{theorem}\label{thm:classification Schur}
Let $S(n,r)$ be the Schur algebra over $\mathbb{F}$.
\begin{description}\setlength{\itemsep}{-3pt}
\item[(1)] If $p=2$, then $S(2,8)$, $S(2,17)$ and $S(2,19)$ are $\tau$-tilting finite.
\item[(2)] If $p=2$, then $S(3,4)$ is $\tau$-tilting finite.
\item[(3)] If $p=2$, then $S(n,5)$ is $\tau$-tilting infinite for any $n\geqslant 5$.
\item[(4)] If $p\geqslant 5$, then $S(2,r)$ is $\tau$-tilting finite for any $p^2\leqslant r\leqslant p^2+p-1$.
\end{description}
\end{theorem}
\begin{proof}
(1) We recall from \cite[Section 4.1]{W-schur} that
\begin{center}
$\overline{S(2,8)}\simeq \mathcal{L}_5$, $\quad$$\overline{S(2,17)}\simeq \mathcal{L}_5 \oplus \mathcal{A}_2 \oplus \mathbb{F}\oplus \mathbb{F}$, $\quad$ $\overline{S(2,19)}\simeq \mathcal{L}_5\oplus \mathcal{D}_3 \oplus\mathbb{F} \oplus \mathbb{F}$.
\end{center}
By Proposition \ref{prop:Am}, Lemma \ref{prop:KLMN} and Proposition \ref{prop:Dm}, we have already known that $\mathcal{A}_2$, $\mathcal{D}_3$ and $\mathcal{L}_5$ are $\tau$-tilting finite. Thus, we conclude that $S(2,8)$, $S(2,17)$ and $S(2,19)$ are $\tau$-tilting finite. More precisely, the number of support $\tau$-tilting modules are $1656$, $39744$ and $185472$, respectively.

(2) We recall from \cite[3.6]{DEMN-tame schur} that $S(3,4)$ over $p=2$ is Morita equivalent to $\mathcal{M}_4$. Then the assertion immediately follows from Lemma \ref{prop:KLMN}.

(3) By the definition of Schur algebras, $S(n,5)$ with $n\geqslant 6$ is always Morita equivalent to $S(5,5)$ over the same field. It is shown in \cite[Proposition 3.8]{Xi-schur} that $S(5,5)$ over $p=2$ is Morita equivalent to $\mathcal{N}_5\oplus \mathcal{A}_2$. Since $\mathcal{N}_5$ is $\tau$-tilting infinite by Lemma \ref{prop:KLMN}, so is $S(n,5)$ with $n\geqslant 5$ over $p=2$.

(4) We look at the case $S(2,r)$ with $p^2\leqslant r\leqslant p^2+p-1$ over $p\geqslant 5$. We use Lemma \ref{EH-quiver-s(2,r)} to understand the quiver of the basic algebra $\overline{S(2,r)}$ of $S(2,r)$. We notice that the number of vertices in the quiver increases regularly when $r$ increases, see Table \ref{quiver-S(2,r)}.
\begin{table}
\centering
\caption{the quiver of $\overline{S(2,r)}$ over $p\geqslant 5$}\label{quiver-S(2,r)}
\begin{center}
$\xymatrix@C=0.7cm@R=0.7cm{
&2p-1&&4p-1 & &p^2-p-1\ar@<0.5ex>[r]^{\ }&p^2+p-1\ar@<0.5ex>[l]_{\ }\\
0\ar@<0.5ex>[r]^{\ }&2p-2 \ar@<0.5ex>[r]^{\ }\ar@<0.5ex>[l]_{\ }&2p \ar@<0.5ex>[r]^{\ }\ar@<0.5ex>[l]_{\ }&4p-2\ar@{.}[r]\ar@<0.5ex>[l]_{\ }&p^2-3p\ar@<0.5ex>[d]_{\ }\ar@<0.5ex>[r]^{\ }&p^2-p-2\ar@<0.5ex>[r]^{\ }\ar@<0.5ex>[l]_{\ }\ar@<0.5ex>[d]_{\ }&p^2-p\ar@<0.5ex>[l]_{\ }\\
 & & & &p^2+p\ar@<0.5ex>[r]^{\ }\ar@<0.5ex>[u]^{\ }&p^2+p-2\ar@<0.5ex>[u]^{\ }\ar@<0.5ex>[l]_{\ }& \\
1\ar@<0.5ex>[r]^{\ }&2p-3 \ar@<0.5ex>[r]^{\ }\ar@<0.5ex>[l]_{\ }&2p+1 \ar@<0.5ex>[r]^{\ }\ar@<0.5ex>[l]_{\ }&4p-3\ar@<0.5ex>[l]_{\ }\ar@{.}[r]&p^2-3p+1\ar@<0.5ex>[d]_{\ }\ar@<0.5ex>[r]^{\ }&p^2-p-3\ar@<0.5ex>[r]^{\ }\ar@<0.5ex>[l]_{\ }\ar@<0.5ex>[d]_{\ }&p^2-p+1\ar@<0.5ex>[l]_{\ }\\
 & & & &p^2+p+1\ar@<0.5ex>[r]^{\ }\ar@<0.5ex>[u]^{\ }&p^2+p-3\ar@<0.5ex>[u]^{\ }\ar@<0.5ex>[l]_{\ }& \\
2\ar@<0.5ex>[r]^{\ }&2p-4 \ar@<0.5ex>[r]^{\ }\ar@<0.5ex>[l]_{\ }&2p+2 \ar@<0.5ex>[r]^{\ }\ar@<0.5ex>[l]_{\ }&4p-4\ar@<0.5ex>[l]_{\ }\ar@{.}[r]&p^2-3p+2\ar@<0.5ex>[d]_{\ }\ar@<0.5ex>[r]^{\ }&p^2-p-4\ar@<0.5ex>[r]^{\ }\ar@<0.5ex>[l]_{\ }\ar@<0.5ex>[d]_{\ }&p^2-p+2\ar@<0.5ex>[l]_{\ }\\
 & & & &p^2+p+2\ar@<0.5ex>[r]^{\ }\ar@<0.5ex>[u]^{\ }&p^2+p-4\ar@<0.5ex>[u]^{\ }\ar@<0.5ex>[l]_{\ }& \\
\vdots&\vdots&\vdots&\vdots&\vdots&\vdots&\vdots\\
p-3\ar@<0.5ex>[r]^{\ }&p+1 \ar@<0.5ex>[r]^{\ }\ar@<0.5ex>[l]_{\ }&3p-3 \ar@<0.5ex>[r]^{\ }\ar@<0.5ex>[l]_{\ }&3p+1\ar@<0.5ex>[l]_{\ }\ar@{.}[r]&p^2-2p-3\ar@<0.5ex>[r]^{\ }&p^2-2p+1\ar@<0.5ex>[r]^{\ }\ar@<0.5ex>[l]_{\ }\ar@<0.5ex>[d]_{\ }&p^2-3\ar@<0.5ex>[l]_{\ }\\
 & & & & &p^2+1\ar@<0.5ex>[u]^{\ }&\\
p-2\ar@<0.5ex>[r]^{\ }&p \ar@<0.5ex>[r]^{\ }\ar@<0.5ex>[l]_{\ }&3p-2 \ar@<0.5ex>[r]^{\ }\ar@<0.5ex>[l]_{\ }&3p\ar@<0.5ex>[l]_{\ }\ar@{.}[r]&p^2-2p-2\ar@<0.5ex>[r]^{\ }&p^2-2p\ar@<0.5ex>[d]_{\ }\ar@<0.5ex>[r]^{\ }\ar@<0.5ex>[l]_{\ }&p^2-2\ar@<0.5ex>[l]_{\ }\\
 & & & & &p^2\ar@<0.5ex>[u]^{\ }&p^2-1  \\
p-1 & &3p-1& &p^2-2p-1& &
}$
\end{center}
\end{table}

Let $\mathcal{B}$ be a block of $\overline{S(2,r)}$. By looking at the quiver of $\overline{S(2,r)}$ as displayed in Table \ref{quiver-S(2,r)}, we find that if $p^2\leqslant r\leqslant p^2+p-1$, then $1\leqslant |\mathcal{B}|\leqslant p+1$. We have the following observations.
\begin{itemize}\setlength{\itemsep}{-3pt}
  \item For $1\leqslant m\leqslant p$, we can find a block $\mathcal{B}'$ of $S(2,r')$ with $1\leqslant r'\leqslant p^2-1$, such that $|\mathcal{B}'|=m$, see \cite[Proposition 5.1]{Erdmann-finite}. Then $\mathcal{B}'$ is Morita equivalent to $\mathbb{F}$ if $m=1$ and $\mathcal{B}'$ is Morita equivalent to $\mathcal{A}_{m}$ if $2\leqslant m$ by Proposition \ref{summary} and Proposition \ref{rep-finite-block}.
  \item We notice from \cite[3.4]{DEMN-tame schur} that $S(2,p^2)$ has a block $\mathcal{B}'$ satisfying $|\mathcal{B}'|=p+1$, which is Morita equivalent to the algebra $\mathcal{D}_{p+1}$.
\end{itemize}
By Lemma \ref{EH-blocks-s(2,r)}, we have $\mathcal{B}\simeq \mathbb{F}$ if $|\mathcal{B}|=1$, $\mathcal{B}\simeq \mathcal{A}_{|\mathcal{B}|}$ if $2\leqslant |\mathcal{B}|\leqslant p$ and $\mathcal{B}\simeq \mathcal{D}_{p+1}$ if $|\mathcal{B}|=p+1$. We conclude that $\overline{S(2,r)}$ with $p^2\leqslant r\leqslant p^2+p-1$ contains only $\mathbb{F}$, $\mathcal{A}_m$ ($2\leqslant m\leqslant p$) and $\mathcal{D}_{p+1}$ as blocks. Since $\mathcal{A}_m$ is obviously $\tau$-tilting finite, the problem in this case is reduced to the $\tau$-tilting finiteness of $\mathcal{D}_{p+1}$. We get the assertion since we have already shown in Proposition \ref{prop:Dm} that $\mathcal{D}_{p+1}$ is $\tau$-tilting finite.
\end{proof}

Now, the $\tau$-tilting finiteness of $S(n,r)$ is completely determined. As a summary, we provide a complete list of $\tau$-tilting finite Schur algebras in Appendix \ref{appendix}. We are able to show the number of support $\tau$-tilting modules for any $\tau$-tilting finite $S(n,r)$ over $p=2,3$, see also Appendix \ref{appendix}.

\subsection{$\tau$-tilting finiteness of blocks of Schur algebras}
As we have seen in the previous subsection, $\tau$-tilting infiniteness of a Schur algebra does not imply $\tau$-tilting infiniteness of its blocks. For instance, the Schur algebra $S(5,5)$ over $p=2$ is $\tau$-tilting infinite but it contains a $\tau$-tilting finite block which is Morita equivalent to $\mathcal{A}_2$, see Theorem \ref{thm:classification Schur} (3). Thus, it is natural to ask the $\tau$-tilting finiteness for blocks of Schur algebras.
\begin{problem*}\label{problem 4.1}
Give a complete classification of $\tau$-tilting finite blocks of Schur algebras.
\end{problem*}

We give a partial answer to the above problem, that is, we completely determine the $\tau$-tilting finiteness for blocks of $S(2,r)$. Many parts of the result have been given in the previous subsection.
\begin{theorem}\label{thm:S_2r}
Let $\mathcal{B}$ be a block of $S(2,r)$. 
\begin{description}\setlength{\itemsep}{-3pt}
\item[(1)] If $p=2$, then $\mathcal{B}$ is $\tau$-tilting finite if and only if $\mathcal{B}$ is Morita equivalent to one of $\mathbb{F}$, $\mathcal{A}_2$, $\mathcal{D}_3$, $\mathcal{K}_4$ and $\mathcal{L}_5$.
\item[(2)] If $p\geqslant 3$, then $\mathcal{B}$ is $\tau$-tilting finite if and only if $\mathcal{B}$ is Morita equivalent to one of $\mathbb{F}$, $\mathcal{A}_m (2\leqslant m \leqslant p)$ and $\mathcal{D}_{p+1}$.
\end{description}
\end{theorem}
\begin{proof}
We have already shown in the previous section that $\mathcal{B}$ is $\tau$-tilting finite if $\mathcal{B}$ is Morita equivalent to one of $\mathbb{F}$, $\mathcal{A}_m (2\leqslant m \leqslant p)$, $\mathcal{K}_4$, $\mathcal{L}_5$ and $\mathcal{D}_{p+1}$. Next, we show the necessity. We denote by $\overline{S(2,r)}$ the basic algebra of $S(2,r)$ and we assume that $\mathcal{B}$ is $\tau$-tilting finite. It is obvious from Lemma \ref{EH-blocks-s(2,r)} that $\mathcal{B}$ is Morita equivalent to $\mathbb{F}$, $\mathcal{A}_2$ if $|\mathcal{B}|=1,2$, respectively. In the following, suppose $|\mathcal{B}|\geqslant 3$.

(1) Let $p=2$. We recall from \cite{Erdmann-finite} that $\overline{S(2,4)}$ is isomorphic to $\mathcal{D}_3$ and $\overline{S(2,6)}$ is isomorphic to $\mathcal{K}_4$. Also, it is given by \cite{W-schur} that $\overline{S(2,8)}$ is isomorphic to $\mathcal{L}_5$. By Lemma \ref{EH-blocks-s(2,r)}, we conclude that $\mathcal{B}$ is Morita equivalent to $\mathcal{D}_3$, $\mathcal{K}_4$, $\mathcal{L}_5$ if $|\mathcal{B}|=3, 4, 5$, respectively. By using Lemma \ref{EH-quiver-s(2,r)}, we find that the quiver of $\overline{S(2,10)}$ is of form
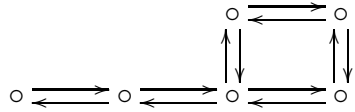
\begin{center}
$\vcenter{\xymatrix@C=1cm@R=0.7cm{ &&\circ \ar@<0.5ex>[d]^{\ }\ar@<0.5ex>[r]^{\ }&\circ \ar@<0.5ex>[l]^{\ }\ar@<0.5ex>[d]^{\ }\\ \circ \ar@<0.5ex>[r]^{\ } &\circ \ar@<0.5ex>[r]^{\ }\ar@<0.5ex>[l]^{\ } &\circ \ar@<0.5ex>[u]^{\ }\ar@<0.5ex>[l]^{\ }\ar@<0.5ex>[r]^{\ }&\circ \ar@<0.5ex>[u]^{\ }\ar@<0.5ex>[l]^{\ }}}$
\end{center}
with $6$ vertices and it contains a $\tau$-tilting infinite subquiver, see Proposition \ref{infinite-square}. Also, according to Lemma \ref{EH-quiver-s(2,r)}, we find that if $|\mathcal{B}|>6$, then the quiver of $\mathcal{B}$ is obtained by adding extra vertices to the above quiver. Hence, $\mathcal{B}$ is $\tau$-tilting infinite if $|\mathcal{B}|\geqslant 6$.

(2) Let $p=3$. We recall from \cite{Erdmann-finite} that $\overline{S(2,6)}$ is isomorphic to $\mathcal{A}_3\oplus \mathbb{F}$ and $\overline{S(2,9)}$ is isomorphic to $\mathcal{D}_4\oplus \mathbb{F}$. By Lemma \ref{EH-blocks-s(2,r)}, we conclude that $\mathcal{B}$ is Morita equivalent to $\mathcal{A}_3$ if $|\mathcal{B}|=3$ and $\mathcal{B}$ is Morita equivalent to $\mathcal{D}_4$ if $|\mathcal{B}|=4$. Similarly, by using Lemma \ref{EH-quiver-s(2,r)}, we find that the quiver of the principal block of $\overline{S(2,12)}$ is of form
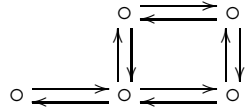
\begin{center}
$\vcenter{\xymatrix@C=1cm@R=0.7cm{&\circ \ar@<0.5ex>[d]^{\ }\ar@<0.5ex>[r]^{\ }&\circ \ar@<0.5ex>[l]^{\ }\ar@<0.5ex>[d]^{\ }\\\circ \ar@<0.5ex>[r]^{\ }&\circ \ar@<0.5ex>[u]^{\ }\ar@<0.5ex>[l]^{\ }\ar@<0.5ex>[r]^{\ }&\circ \ar@<0.5ex>[u]^{\ }\ar@<0.5ex>[l]^{\ }}}$
\end{center}
with $5$ vertices and it contains a $\tau$-tilting infinite subquiver. Therefore, $\mathcal{B}$ is $\tau$-tilting infinite if $|\mathcal{B}|\geqslant 5$.

(3) Let $p\geqslant 5$. We have explained in the proof of Theorem \ref{thm:classification Schur} that if $|\mathcal{B}|=m$ with $3\leqslant m\leqslant p$, then $\mathcal{B}$ is Morita equivalent to $\mathcal{A}_m$; if $|\mathcal{B}|=p+1$, then $\mathcal{B}$ is Morita equivalent to $\mathcal{D}_{p+1}$. It is not difficult to find in Table \ref{quiver-S(2,r)} that if $|\mathcal{B}|\geqslant p+2$, then $\mathcal{B}$ must contain
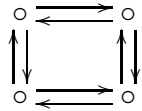
\begin{center}
$\xymatrix@C=1cm@R=0.7cm{\circ \ar@<0.5ex>[d]^{\ }\ar@<0.5ex>[r]^{\ }&\circ \ar@<0.5ex>[l]^{\ }\ar@<0.5ex>[d]^{\ }\\ \circ \ar@<0.5ex>[u]^{\ }\ar@<0.5ex>[r]^{\ }&\circ \ar@<0.5ex>[u]^{\ }\ar@<0.5ex>[l]^{\ }}$
\end{center}
as a subquiver. This implies that $\mathcal{B}$ is $\tau$-tilting infinite if $|\mathcal{B}|\geqslant p+2$.
\end{proof}

\section{Strictly wild Schur algebras}
A finite-dimensional $\mathbb{F}$-algebra $A$ is said to be \emph{strictly wild} if there exists a fully faithful exact $\mathbb{F}$-linear functor from $\mod B$ to $\mod A$, for any finite-dimensional algebra $B$. 
It is obvious that a strictly wild algebra is wild. While a nice characterization of the strictly wildness of hereditary algebras is known, it is not easy to check the strictly wildness for an arbitrary bound quiver algebra. We review the known characterization as follows. 

\begin{proposition}[{\cite[XVIII Corollary 4.7]{SS07}}]\label{prop:Dm}
A path algebra $\mathbb{F}Q$ is strictly wild if and only if the underlying graph of $Q$ is neither Dynkin nor Euclidean.
\end{proposition}

A right $A$-module $M$ is called a \emph{brick} (or \emph{Schurian module}) if $\mathrm{End}_A M\simeq\mathbb{F}$. Then, $A$ is said to be \emph{brick-finite} if it has finitely many isomorphism classes of bricks, and \emph{brick-infinite} otherwise. A brick-infinite algebra $A$ is called \emph{brick-tame} if all but finitely many bricks can be organized in finitely many one-parameter families, in each dimension, and $A$ is called \emph{brick-wild} otherwise. We then have the following observation on strictly wildness for an arbitrary bound quiver algebra.

\begin{proposition}
Let $A$ be a bound quiver algebra. If $A$ is $\tau$-tilting finite, then $A$ is not strictly wild.
\end{proposition}
\begin{proof}
This relies on the so-called \emph{brick-$\tau$-rigid correspondence} introduced in \cite{DIJ-tau-tilting-finite}, that is, there is a bijection between indecomposable $\tau$-rigid modules (up to isomorphism) and bricks $M$ in $\mod A$ such that the smallest torsion class in $\mod A$ containing $M$ is functorially finite. It turns out that $A$ is $\tau$-tilting finite if and only if $A$ is brick-finite. If $A$ is strictly wild, then $A$ is brick-wild as mentioned in \cite[Remark 1]{CC-Schur-tame}, and hence $A$ is $\tau$-tilting infinite.
\end{proof}

We note that strictly wildness behaves well under quotient and idempotent truncation techniques. If $S(n,r)$ is strictly wild, then both $S(N,r)$ for $N>n$ and $S(n,r+n)$ are strictly wild. In the following, we look at some edge cases where all data can be found in \cite{W-schur} and in this paper.
\begin{description}\setlength{\itemsep}{-3pt}
\item[(1)] Set $p=2$. Then, $S(2,r)$ with $r=6,8,13,15,17,19$, $S(3,r)$ with $r=4,5$, $S(4,5)$, are wild and $\tau$-tilting finite, and hence each of them is not strictly wild; $S(2,r)$ with $r=10,21$, $S(3,r)$ with $r=6,7$, $S(4,4)$, are strictly wild since each of them contains 
\begin{center}
$\Delta_5: \quad \vcenter{\xymatrix@C=1cm@R=0.7cm{&\circ \ar@<0.5ex>[d] \ar@<0.5ex>[r] &\circ \ar@<0.5ex>[l] \ar@<0.5ex>[d] \\\circ \ar@<0.5ex>[r] &\circ \ar@<0.5ex>[u] \ar@<0.5ex>[l] \ar@<0.5ex>[r] &\circ \ar@<0.5ex>[u] \ar@<0.5ex>[l] }}$
\end{center}
as a subquiver; $S(3,8)$ is strictly wild since it contains 
\begin{center}
$\Delta_7: \quad \vcenter{\xymatrix@C=1.2cm@R=0.7cm{\circ \ar@<0.5ex>[d]  &\circ  \ar@<0.5ex>[d] \\
\circ \ar@<0.5ex>[u] \ar@<0.5ex>[d] \ar@<0.5ex>[r] &\circ \ar@<0.5ex>[u] \ar@<0.5ex>[d]\ar@<0.5ex>[l]\ar@<0.5ex>[r] &\circ\ar@<0.5ex>[l]\\
\circ\ar@<0.5ex>[u]&\circ\ar@<0.5ex>[u]
}}$
\end{center}
as a subquiver; $S(5,5)$ is wild and $\tau$-tilting infinite, the basic algebra of $S(5,5)$ is $\mathcal{N}_5\oplus \mathcal{A}_2$ whose quiver is displayed as 
\begin{center}
$\xymatrix@C=1cm@R=0.3cm{\circ\ar@<0.5ex>[r] &\circ\ar@<0.5ex>[l] \ar@<0.5ex>[r] &\circ\ar@<0.5ex>[l] \ar@<0.5ex>[r] &
\circ\ar@<0.5ex>[l] \ar@<0.5ex>[r] &\circ\ar@<0.5ex>[l]&
\circ \ar@<0.5ex>[r] &\circ\ar@<0.5ex>[l]}$,
\end{center}
our method is invalid in this case.

\item[(2)] Set $p=3$. Then, $S(2,r)$ with $r=12,13$, $S(3,6)$, $S(4,r)$ with $r=7,8$, are strictly wild since each of them contains $\Delta_5$ as a subquiver; using quotient and idempotent truncation techniques, one may find $\Delta_5$ as a subquiver in the quiver of certain idempotent truncation (resp. quotient) of $S(4,10)$ (resp. $S(4,10)$), hence these two are strictly wild; $S(3,r)$ with $r=10,11$ is wild and $\tau$-tilting infinite, our method is invalid similar to the case $S(5,5)$ over $p=2$.

\item[(3)] Set $p\ge 5$. Then, $S(2,r)$ with $p^2\leqslant r\leqslant p^2+p-1$ is wild and $\tau$-tilting finite, so it is not strictly wild; $S(2,p+p^2)$ and $S(3,r)$ with $r=2p, 2p+1, 2p+2$ are strictly wild since both of them contain $\Delta_5$ as a subquiver.
\end{description}
In summary, we have known the strictly wildness of $S(n,r)$ for any characteristic $p\ge 2$, except for $p=2, n\ge 5, r=5$ and $p=3, n=3, r=10,11$.

\appendix
\section{A complete list of $\tau$-tilting finite Schur algebras}\label{appendix}
Since $S(n,r)$ is an idempotent truncation (resp. a quotient) of $S(N,r)$ for any $N>n$ (resp. $S(n,n+r)$), both $S(N,r)$ and $S(n,n+r)$ are $\tau$-tilting infinite if  $S(n,r)$ is $\tau$-tilting infinite. Thus, we only need to consider the cases with small $n$ and $r$. It needs two steps to give a complete list of $\tau$-tilting finite Schur algebras. The first step is to construct the basic algebra of $S(n,r)$ with small $n$ and $r$, while the second step is to check the $\tau$-tilting finiteness of these basic algebras. One may gradually enlarge $n$ and $r$, and repeat the second step until one can find a complete classification.

As we mentioned in the introduction, there is nothing new in this paper toward the first step since the work in \cite{Erdmann-finite}, \cite{DEMN-tame schur} and \cite{W-schur} provide enough materials for this paper. We have already introduced most of the needed algebras in the previous section, but we still need the following three cases. We recall from \cite{Erdmann-finite} that
\begin{itemize}\setlength{\itemsep}{-3pt}
\item $\mathcal{U}_4:=\mathbb{F}Q/I$ is the bound quiver algebra given by
\begin{center}
$Q: \xymatrix@C=1cm{1\ar@<0.5ex>[r]^{\alpha_1}&2\ar@<0.5ex>[l]^{\beta_1}\ar@<0.5ex>[r]^{\alpha_2}&3\ar@<0.5ex>[l]^{\beta_2}\ar@<0.5ex>[r]^{\alpha_3}&
4\ar@<0.5ex>[l]^{\beta_3}}$ and $I: \left \langle \begin{matrix}
\alpha_1\beta_1, \alpha_2\beta_2, \alpha_1\alpha_2\alpha_3,\beta_3\beta_2\beta_1,\alpha_3\beta_3-\beta_2\alpha_2
\end{matrix}\right \rangle$;
\end{center}

\item $\mathcal{R}_4:=\mathbb{F}Q/I$ is the bound quiver algebra given by
\begin{center}
$Q:\xymatrix@C=1cm{1\ar@<0.5ex>[r]^{\alpha_1}&2\ar@<0.5ex>[l]^{\beta_1}\ar@<0.5ex>[r]^{\alpha_2}&3\ar@<0.5ex>[l]^{\beta_2}
\ar@<0.5ex>[r]^{\alpha_3}&4\ar@<0.5ex>[l]^{\beta_3}}$ and $I: \left \langle\begin{matrix}
\alpha_1\beta_1,\alpha_1\alpha_2, \beta_2\beta_1,\\
\alpha_2\beta_2-\beta_1\alpha_1, \alpha_3\beta_3-\beta_2\alpha_2
\end{matrix}\right \rangle$;
\end{center}

\item $\mathcal{H}_4:=\mathbb{F}Q/I$ is the bound quiver algebra given by
\begin{center}
$Q: \xymatrix@C=1cm@R=0.8cm{1\ar@<0.5ex>[r]^{\alpha_1}&2\ar@<0.5ex>[l]^{\beta_1}\ar@<0.5ex>[d]^{\beta_2}\ar@<0.5ex>[r]^{\alpha_3}&4\ar@<0.5ex>[l]^{\beta_3}\\
&3\ar@<0.5ex>[u]^{\alpha_2}&}$ and $I:  \left \langle \begin{matrix}
\alpha_1\beta_1,\alpha_1\beta_2,\alpha_2\beta_1,\alpha_2\beta_2,\alpha_1\alpha_3,\\
\beta_3\beta_1, \alpha_3\beta_3-\beta_1\alpha_1-\beta_2\alpha_2
\end{matrix}\right \rangle$.
\end{center}
\end{itemize}

Toward the second step, we get some new results in this paper. For the convenience of readers, we recall some facts as follows. Here, the numbers for $\mathcal{R}_4, \mathcal{H}_4, \mathcal{U}_4$ have been given in \cite[Lemma 3.3]{W-schur} and \cite[Proposition 4.4]{W-schur}.
\begin{center}
\begin{tabular}{|c|c|c|c|c|c|c|c|c|c|c|c|c|}
\hline
$A$ &$\mathbb{F}$ &$\mathcal{A}_2$ & $\mathcal{A}_3$ & $\mathcal{D}_3$ &$\mathcal{R}_4$&$\mathcal{H}_4$&$\mathcal{D}_4$ & $\mathcal{K}_4$ &$\mathcal{U}_4$ & $\mathcal{M}_4$ & $\mathcal{L}_5$ \\
\hline
$\#\stautilt A$ &2& 6&20&28&88&96&114&136 &136 &  152  &  1656 \\ \hline
\end{tabular}
\end{center}
In fact, one may refer to \href{https://infinite-wang.github.io/Notes/}{https://infinite-wang.github.io/Notes/} for a complete list of $g$-vectors of the above algebras.

We use some visual tables to display the complete classification. In the following tables, we claim that the color \textbf{blue} means $\tau$-tilting finite, the color \textbf{orange} means $\tau$-tilting infinite, the capital letter \textbf{S} means semi-simple, the capital letter \textbf{F} means representation-finite, the capital letter \textbf{T} means tame and the capital letter \textbf{W} means wild. In particular, we use Proposition \ref{prop:number-product} to calculate the number $\#\stautilt S(n,r)$ for a $\tau$-tilting finite $S(n,r)$.

\subsection{The $\tau$-tilting finiteness of $S(n,r)$ over $p=2$}
\begin{center}
\scalebox{0.8}{
\begin{tabular}{|c|c|c|c|c|c|c|c|c|c|c|c|c|c|c|c|c|c|c|c|c|c|c}
\hline
\diagbox[height=2em]{$n$}{$r$}&1&2&3&4&5&6&7&8&9&10&11&12&13&14&15&16&17&18&19&20&21&$\cdots$\\ \hline
2&\cellcolor{myblue}S&\cellcolor{myblue}F&\cellcolor{myblue}S&\cellcolor{myblue}T&\cellcolor{myblue}F&\cellcolor{myblue}W&\cellcolor{myblue}F&\cellcolor{myblue}W
&\cellcolor{myblue}T&\cellcolor{myred}W&\cellcolor{myblue}T&\cellcolor{myred}W&\cellcolor{myblue}W&\cellcolor{myred}W&\cellcolor{myblue}W&\cellcolor{myred}W&\cellcolor{myblue}W&\cellcolor{myred}W&\cellcolor{myblue}W&
\cellcolor{myred}W&\cellcolor{myred}W&\cellcolor{myred}$\cdots$ \\ \hline
\end{tabular}}
\end{center}
\begin{center}
\renewcommand\arraystretch{1}
\begin{tabular}{|c|c|c|c|c|c|c|c|c|c|c|c|c|c|c}
\hline
\diagbox[height=2em]{$n$}{$r$}&1&2&3&4&5&6&7&8&9&10&11&12&13&$\cdots$\\ \hline
3&\cellcolor{myblue}S&\cellcolor{myblue}F&\cellcolor{myblue}F&\cellcolor{myblue}W&\cellcolor{myblue}W&\cellcolor{myred}W&\cellcolor{myred}W&\cellcolor{myred}W&\cellcolor{myred}W&\cellcolor{myred}W&
\cellcolor{myred}W&\cellcolor{myred}W&\cellcolor{myred}W&\cellcolor{myred}$\cdots$\\ \hline
4&\cellcolor{myblue}S&\cellcolor{myblue}F&\cellcolor{myblue}F&\cellcolor{myred}W&\cellcolor{myblue}W&\cellcolor{myred}W&\cellcolor{myred}W&\cellcolor{myred}W&\cellcolor{myred}W&\cellcolor{myred}W&
\cellcolor{myred}W&\cellcolor{myred}W&\cellcolor{myred}W&\cellcolor{myred}$\cdots$\\ \hline
5&\cellcolor{myblue}S&\cellcolor{myblue}F&\cellcolor{myblue}F&\cellcolor{myred}W&\cellcolor{myred}W&\cellcolor{myred}W&\cellcolor{myred}W&\cellcolor{myred}W&\cellcolor{myred}W&\cellcolor{myred}W&
\cellcolor{myred}W&\cellcolor{myred}W&\cellcolor{myred}W&\cellcolor{myred}$\cdots$\\ \hline
6&\cellcolor{myblue}S&\cellcolor{myblue}F&\cellcolor{myblue}F&\cellcolor{myred}W&\cellcolor{myred}W&\cellcolor{myred}W&\cellcolor{myred}W&\cellcolor{myred}W&\cellcolor{myred}W&\cellcolor{myred}W&
\cellcolor{myred}W&\cellcolor{myred}W&\cellcolor{myred}W&\cellcolor{myred}$\cdots$\\ \hline
$\vdots$&\cellcolor{myblue}$\vdots$&\cellcolor{myblue}$\vdots$&\cellcolor{myblue}$\vdots$&\cellcolor{myred}$\vdots$&\cellcolor{myred}$\vdots$&\cellcolor{myred}$\vdots$&
\cellcolor{myred}$\vdots$&\cellcolor{myred}$\vdots$&\cellcolor{myred}$\vdots$&\cellcolor{myred}$\vdots$&\cellcolor{myred}$\vdots$&\cellcolor{myred}
$\vdots$&\cellcolor{myred}$\vdots$&\cellcolor{myred}$\ddots$
\end{tabular}
\end{center}
We list all $\tau$-tilting finite $S(n,r)$ over $p=2$ as follows.
\begin{center}
\renewcommand\arraystretch{1}
\begin{tabular}{|c|c|c|c|}
\hline
 $S(n,r)$ & The basic algebra of $S(n,r)$ & Morita equivalence &$\#\stautilt S(n,r)$\\
\hline
 $S(2,1)$ & $\mathbb{F}$ &  $\simeq S(n,1)$ for any $n\geqslant3$ &2 \\
 $S(2,2)$ & $\mathcal{A}_2$ &  $\simeq S(n,2)$ for any $n\geqslant3$ &6 \\
 $S(2,3)$ & $\mathbb{F}\oplus \mathbb{F}$ &  &4  \\
 $S(2,4)$ & $\mathcal{D}_3$ &   &28 \\
 $S(2,5)$ & $\mathcal{A}_2\oplus \mathbb{F}$  & &12  \\
 $S(2,6)$ & $\mathcal{K}_4$  &   &136 \\
 $S(2,7)$ & $\mathcal{A}_2\oplus \mathbb{F}\oplus \mathbb{F}$  &  &24  \\
 $S(2,8)$ & $\mathcal{L}_5$  &  & 1656 \\
 $S(2,9)$ & $\mathcal{D}_3\oplus\mathbb{F}\oplus \mathbb{F}$ & &112  \\
 $S(2,11)$ & $\mathcal{D}_3\oplus\mathcal{A}_2\oplus  \mathbb{F}$ && 336   \\
 $S(2,13)$ & $\mathcal{K}_4\oplus \mathcal{A}_2\oplus \mathbb{F}$ && 1632   \\
 $S(2,15)$ & $\mathcal{K}_4\oplus \mathcal{A}_2\oplus \mathbb{F}\oplus \mathbb{F}$ &  &3264  \\
 $S(2,17)$ & $\mathcal{L}_5\oplus \mathcal{A}_2\oplus \mathbb{F}\oplus \mathbb{F}$ &  &39744  \\
 $S(2,19)$ & $\mathcal{L}_5\oplus \mathcal{D}_3\oplus \mathbb{F}\oplus \mathbb{F}$ &  &185472  \\
 $S(3,3)$ & $\mathcal{A}_2\oplus \mathbb{F}$ & $\simeq S(n,3)$ for any $n\geqslant4$ &12 \\
 $S(3,4)$ & $\mathcal{M}_4$ &  &152  \\
 $S(3,5)$ & $\mathcal{U}_4$ & &136   \\
 $S(4,5)$ & $\mathcal{U}_4\oplus \mathcal{A}_2$ & &816 \\
 \hline
\end{tabular}
\end{center}

\subsection{The $\tau$-tilting finiteness of $S(n,r)$ over $p=3$}
\begin{center}
\renewcommand\arraystretch{1}
\begin{tabular}{|c|c|c|c|c|c|c|c|c|c|c|c|c|c|c}
\hline
\diagbox[height=2em]{$n$}{$r$}&1&2&3&4&5&6&7&8&9&10&11&12&13&$\cdots$\\ \hline
2&\cellcolor{myblue}S&\cellcolor{myblue}S&\cellcolor{myblue}F&\cellcolor{myblue}F&\cellcolor{myblue}F&\cellcolor{myblue}F&\cellcolor{myblue}F&\cellcolor{myblue}F&
\cellcolor{myblue}T&\cellcolor{myblue}T&\cellcolor{myblue}T&\cellcolor{myred}W&\cellcolor{myred}W&\cellcolor{myred}$\cdots$\\ \hline
3&\cellcolor{myblue}S&\cellcolor{myblue}S&\cellcolor{myblue}F&\cellcolor{myblue}F&\cellcolor{myblue}F&\cellcolor{myred}W&\cellcolor{myblue}T&
\cellcolor{myblue}T&\cellcolor{myred}W&\cellcolor{myred}W&\cellcolor{myred}W&\cellcolor{myred}W&\cellcolor{myred}W&\cellcolor{myred}$\cdots$ \\ \hline
4&\cellcolor{myblue}S&\cellcolor{myblue}S&\cellcolor{myblue}F&\cellcolor{myblue}F&\cellcolor{myblue}F&\cellcolor{myred}W&\cellcolor{myred}W&\cellcolor{myred}W&\cellcolor{myred}W&\cellcolor{myred}W&\cellcolor{myred}W&\cellcolor{myred}W&
\cellcolor{myred}W&\cellcolor{myred}$\cdots$ \\ \hline
5&\cellcolor{myblue}S&\cellcolor{myblue}S&\cellcolor{myblue}F&\cellcolor{myblue}F&\cellcolor{myblue}F&\cellcolor{myred}W&\cellcolor{myred}W&\cellcolor{myred}W&\cellcolor{myred}W&\cellcolor{myred}W&\cellcolor{myred}W&\cellcolor{myred}W&
\cellcolor{myred}W&\cellcolor{myred}$\cdots$ \\ \hline
$\vdots$&\cellcolor{myblue}$\vdots$&\cellcolor{myblue}$\vdots$&\cellcolor{myblue}$\vdots$&\cellcolor{myblue}$\vdots$&\cellcolor{myblue}$\vdots$&\cellcolor{myred}$\vdots$&\cellcolor{myred}$\vdots$&\cellcolor{myred}$\vdots$&\cellcolor{myred}$\vdots$&\cellcolor{myred}$\vdots$&
\cellcolor{myred}$\vdots$&\cellcolor{myred}$\vdots$&\cellcolor{myred}$\vdots$&\cellcolor{myred}$\ddots$
\end{tabular}
\end{center}
We list all $\tau$-tilting finite $S(n,r)$ over $p=3$ as follows.
\begin{center}
\renewcommand\arraystretch{1}
\begin{tabular}{|c|c|l|c|}
\hline
$S(n,r)$ & The basic algebra of $S(n,r)$ & Morita equivalence &$\#\stautilt S(n,r)$\\
\hline
 $S(2,1)$ & $\mathbb{F}$ & $\simeq S(n,1)$ for any $n\geqslant3$ &2\\
 $S(2,2)$ & $\mathbb{F}\oplus \mathbb{F}$ & $\simeq S(n,2)$ for any $n\geqslant3$&4 \\
 $S(2,3)$ & $\mathcal{A}_2$ &    &6 \\
 $S(2,4)$ & $\mathcal{A}_2\oplus \mathbb{F}$ & $\simeq S(2,5)$  &12  \\
 $S(2,6)$ & $\mathcal{A}_3\oplus \mathbb{F}$ & $\simeq S(2,7)$ &40 \\
 $S(2,8)$ & $\mathcal{A}_3\oplus \mathbb{F}\oplus \mathbb{F}$ &  &80 \\
 $S(2,9)$ & $\mathcal{D}_4\oplus \mathbb{F}$ &  &228  \\
 $S(2,10)$ & $\mathcal{D}_4\oplus \mathbb{F}\oplus \mathbb{F}$ & &456   \\
 $S(2,11)$ & $\mathcal{D}_4\oplus \mathcal{A}_2$ &  &684  \\
 $S(3,3)$ & $\mathcal{A}_3$ & $\simeq S(n,3)$ for any $n\geqslant4$  &20 \\
 $S(3,4)$ & $\mathcal{A}_2\oplus \mathbb{F}\oplus \mathbb{F}$ & &24  \\
 $S(3,5)$ & $\mathcal{A}_2\oplus\mathcal{A}_2\oplus \mathbb{F}$ &  &72  \\
 $S(3,7)$ & $\mathcal{R}_4\oplus\mathcal{A}_2\oplus\mathcal{A}_2$ &  &3168  \\
 $S(3,8)$ & $\mathcal{R}_4\oplus\mathcal{H}_4\oplus\mathcal{A}_2$ &  &50688  \\
 $S(4,4)$ & $\mathcal{A}_3\oplus \mathbb{F}\oplus \mathbb{F}$ & $\simeq S(n,4)$ for any $n\geqslant5$ &80    \\
 $S(4,5)$ & $\mathcal{A}_3\oplus \mathcal{A}_2\oplus \mathbb{F}$ &   &240 \\
 $S(5,5)$ & $\mathcal{A}_3\oplus \mathcal{A}_3\oplus \mathbb{F}$ & $\simeq S(n,5)$ for any $n\geqslant6$ &800  \\
 \hline
\end{tabular}
\end{center}

\subsection{The $\tau$-tilting finiteness of $S(n,r)$ over $p\geqslant 5$}
\begin{center}
\renewcommand\arraystretch{1}
\begin{tabular}{|c|c|c|c|c|c|c|c|c|c|c|c|c|c|c}
\hline
\diagbox[height=2em]{$n$}{$r$}&$1\sim p-1$&$p\sim 2p-1$&$2p\sim p^2-1$&$p^2\sim p^2+p-1$&$p^2+p\sim \infty$\\ \hline
2&\cellcolor{myblue}S&\cellcolor{myblue}F&\cellcolor{myblue}F&\cellcolor{myblue}W&\cellcolor{myred}W\\ \hline
3&\cellcolor{myblue}S&\cellcolor{myblue}F&\cellcolor{myred}W&\cellcolor{myred}W&\cellcolor{myred}W\\ \hline
4&\cellcolor{myblue}S&\cellcolor{myblue}F&\cellcolor{myred}W&\cellcolor{myred}W&\cellcolor{myred}W \\ \hline
5&\cellcolor{myblue}S&\cellcolor{myblue}F&\cellcolor{myred}W&\cellcolor{myred}W&\cellcolor{myred}W\\ \hline
$\vdots$&\cellcolor{myblue}$\vdots$&\cellcolor{myblue}$\vdots$&\cellcolor{myred}$\vdots$&\cellcolor{myred}$\vdots$&\cellcolor{myred}$\vdots$
\end{tabular}
\end{center}

Depending on the prime number $p$, we can determine the structure for the basic algebra $\overline{S(n,r)}$ of $\tau$-tilting finite $S(n,r)$ as follows.
\begin{itemize}\setlength{\itemsep}{-3pt}
\item In the case of $n\geqslant 2$ and $1\leqslant r\leqslant p-1$, $\overline{S(n,r)}$ is isomorphic to a finite direct sum of copies of $\mathbb{F}$.
\item In the case of $n\geqslant 2$ and $p\leqslant r\leqslant 2p-1$, $\overline{S(n,r)}$ is isomorphic to a finite direct sum of copies of $\mathbb{F}$ and copies of $\mathcal{A}_m$ with $2\leqslant m\in \mathbb{N}$.
\item In the case of $n=2$ and $2p\leqslant r\leqslant p^2+p-1$, $\overline{S(2,r)}$ is isomorphic to a finite direct sum of copies of $\mathbb{F}$, copies of $\mathcal{A}_m (2\leqslant m \leqslant p)$ and copies of $\mathcal{D}_{p+1}$.
\end{itemize}
In fact, the multiplicities of $\mathbb{F}$, $\mathcal{A}_m$ and $\mathcal{D}_{p+1}$ in each case are determined by the decomposition matrix of the symmetric group $G_r$ over $p$, which is still open.

\newpage
\noindent\textit{Acknowledgements.}
TA is partially supported by JSPS Grants-in-Aid for Scientific Research JP19J11408 and JSPS Grant-in-Aid for Transformative Research Areas (A) 22H05105. QW is partially supported by JSPS Grant-in-Aid for JSPS Fellows 20J10492 and Fundamental Research Funds for the Central Universities DUT25RC(3)132. 
The authors are grateful to Professor Susumu Ariki for his helpful discussions and comments, and thanks to Liron Speyer for suggesting that we consider the strictly wildness of Schur algebras.

\ \\

TA: Graduate School of Human Development and Environment, Kobe University, 3-11 Tsurukabuto,
Nada-ku, Kobe 657-8501, Japan

\emph{Email address}: \texttt{toshitaka.aoki@people.kobe-u.ac.jp}

\vspace{0.3cm}
QW: School of Mathematical Sciences, Dalian University of Technology, Dalian City, Liaoning Province, 116024, China 

\emph{Email address}: \texttt{wang2025@dlut.edu.cn (cc:infinite-wang@outlook.com)}

\end{document}